\documentclass{amsart}
\usepackage[letterpaper, margin=1in]{geometry}

\usepackage{amsmath,amssymb,amsthm}
\usepackage{arydshln} 
\usepackage{dsfont}
\usepackage{enumitem}
\usepackage{graphicx}
\usepackage{hyperref}
\hypersetup{
    colorlinks,
    linkcolor={red!60!black},
    citecolor={blue!60!black},
    urlcolor={green!60!black}
}
\usepackage{cleveref}
\usepackage{mathabx} 
\usepackage{mathrsfs} 
\usepackage{mdframed} 
\usepackage{multicol,multirow}
\usepackage{subcaption}
\usepackage{soul}
\usepackage[svgnames]{xcolor} 

\usepackage{accents} 

\newcommand{\norm}[1]{\left\lVert#1\right\rVert}
\newcommand{\lp}{\left(}
\newcommand{\rp}{\right)}
\newcommand{\ls}{\left[}
\newcommand{\rs}{\right]}
\newcommand{\lv}{\left\vert}
\newcommand{\rv}{\right\vert}
\newcommand{\lc}{\left\{}
\newcommand{\rc}{\right\}}

\newcommand{\HS}{\mathbb{H}^3} 
\newcommand{\R}{\mathbb{R}} 
\newcommand{\Z}{\mathbb{Z}} 

\newcommand{\A}{\mathcal{A}} 
\newcommand{\Q}{\mathcal{Q}} 
\newcommand{\J}{\mathcal{J}} 
\newcommand{\C}{\mathcal{C}} 

\newcommand{\Mend}{M_\mathrm{end}}
\newcommand{\ext}[1]{#1_\mathrm{ext}} 

\newcommand{\til}{\widetilde}
\newcommand{\I}{\mathcal{I}} 
\newcommand{\gt}{\til g}
\newcommand{\Et}{\til\E}
\newcommand{\gradt}{\til\grad}

\newcommand{\ax}{\breve}

\newcommand{\G}{\mathcal{G}} 
\newcommand{\V}{\mathcal{V}} 
\newcommand{\tPhi}{\tilde\Phi} 
\newcommand{\tOmega}{\til\Omega} 
\newcommand{\tSigma}{\til\Sigma} 
\newcommand{\reg}[2]{#1_{\mathrm{reg}(#2)}} 
\newcommand{\bun}[1]{\underaccent{\bar}#1} 

\newcommand{\grad}{\nabla}
\newcommand{\del}{\partial}
\DeclareMathOperator{\Div}{div}
\DeclareMathOperator{\tr}{tr}
\DeclareMathOperator{\Ric}{Ric}

\DeclareMathOperator{\Vol}{\mathrm{Vol}}

\newcommand{\charge}{\widehat}
\newcommand{\E}{\mathcal{E}} 
\newcommand{\gc}{\charge g} 
\newcommand{\gradc}{\charge\nabla} 
\newcommand{\Deltac}{\charge\Delta} 

\newcommand{\jang}{\overline}
\newcommand{\End}{\mathcal{N}}
\newcommand{\gradj}{\jang\grad}
\newcommand{\qj}{\jang q}
\newcommand{\Mj}{\jang M}
\newcommand{\gj}{\jang g}
\newcommand{\Ej}{\jang \E}
\newcommand{\Qj}{\jang \Q}
\newcommand{\Aj}{\jang A}
\newcommand{\Vj}{\jang V}

\newtheorem{theorem}{Theorem}[section]
\newtheorem{lemma}{Lemma}[section]
\newtheorem{proposition}{Proposition}[section]
\newtheorem{corollary}{Corollary}[lemma]
\newtheorem{remark}{Remark}

\usepackage[
    backend=biber,
    style=alphabetic,
    sorting=nyt
]{biblatex}
\allowdisplaybreaks 
\numberwithin{equation}{section}

\usepackage{microtype}

\title{The suboptimal Penrose inequality for charged initial data sets}
\author{Eunice Ng}
\date{}

\begin{document}

\begin{abstract}
    Given a complete, 3-dimensional, asymptotically flat initial data set for the Einstein-Maxwell equations with vanishing magnetic field, we show that there exists a small universal constant $\C$ such that $m\geq\mathcal{C}(\sqrt{\mathcal{A}/(16\pi)}+\mathcal{Q}^2\sqrt{\pi/\mathcal{A}}\,)$. An analogous statement holds for asymptotically hyperboloidal initial data sets with the ADM mass replaced by hyperbolic energy. Under the additional assumption of axisymmetry, the inequality may be further strengthened to include angular momentum. These results extend the Penrose-type inequalities recently obtained by Allen-Bryden-Kazaras-Khuri to the charged setting.
\end{abstract}

\maketitle

\setcounter{tocdepth}{3}

\section{Introduction}\label{sec:intro}

A \emph{charged initial data set} $(M, g, k,\E)$ consists of a 3-dimensional connected and orientable manifold $M$, a complete Riemannian metric $g$, a symmetric 2-tensor $k$, and a vector field $\E$. These quantities are assumed to be smooth and to satisfy the Einstein-Maxwell constraint equations
\begin{equation}
    \begin{aligned}
        16\pi\mu_E &= R_g  + (\tr_g k)^2 - |k|_g^2 - 2|\E|_g^2, \\ 
        8\pi J &= \Div_g(k-(\tr_gk)g), \\
        \Div_g\E &= 0, 
    \end{aligned}
\end{equation}
where $(\mu_E,J)$ is the energy-momentum density of the matter fields and $R_g$ denotes the scalar curvature of $g$. The \emph{charged dominant energy condition} holds if $\mu_E \geq|J|_g$. Furthermore, the initial data is \emph{asymptotically flat with $\ell_0$ ends} if there is a compact set $K$ such that $M\,\setminus\, K = \smash{\bigsqcup_{\ell=1}^{\ell_0} \Mend^\ell}$, where each end $\Mend^\ell$ is diffeomorphic to the complement of a Euclidean 3-ball and admit coordinates $x^i$ such that
\begin{equation}\label{cond:AF_asymptotics}
    \begin{gathered}
        g_{ij} = \delta_{ij} +O_2(|x|^{-q}), \qquad k_{ij} = O_1(|x|^{-2-q}), \qquad \tr_g k = O(|x|^{-1-2q}), \\
        \mu_E, J^i = O(|x|^{-2-2q}), \qquad \E^i = O_1(|x|^{-2}),
    \end{gathered}
\end{equation}
where $q>\frac{1}{2}$ and $O_l(|x|^{-q})$ represents a tensor in the weighted space $C^l_{-q}(\R^3)$. The decay of $\tr_g k$ is not typically assumed but shall be used in our analysis of the Jang equation.

An initial data set is \emph{axisymmetric} if the isometry group of $(M,g)$ contains a $U(1)$ subgroup and all other quantities defining the data are invariant under the $U(1)$ action. That is, if $\eta$ is the Killing field generating the $U(1)$ symmetry, then
\begin{align}
    \mathcal{L}_\eta g = \mathcal{L}_\eta k = \mathcal{L}_\eta\E =0,
\end{align}
where $\mathcal{L}$ denotes Lie differentiation. The decay in \eqref{cond:AF_asymptotics} guarantees that the \emph{ADM energy}, \emph{ADM linear momentum}, \emph{total charge}, and \emph{ADM angular momentum} of each end are well-defined by the limits
\begin{align}
    E_\text{ADM} &= \lim_{r\to\infty} \frac{1}{16\pi} \int_{S_r} \sum_{i=1}^3 (g_{ij,i}-g_{ii,j})\nu^j \, dA, \qquad 
    && P_i = \lim_{r\to\infty} \frac{1}{8\pi} \int_{S_r} (k_{ij}-(\tr_g k)g_{ij})\nu^j \, dA \label{def:ADMenergy} \\
    \Q &= \lim_{r\to\infty} \frac{1}{4\pi} \int_{S_r} \E_i \nu^i \,dA, \qquad
    && \mathcal{J} = \lim_{r\to\infty} \frac{1}{8\pi} \int_{S_r} (k_{ij}-(\tr_gk)g_{ij})\eta^i \nu^j \,dA,\label{def:chargeangularmomentum}
\end{align}
where $S_r$ are coordinate spheres with unit outer normal $\nu$ and area element $dA$. If the initial data is axisymmetric, we shall further assume that the momentum density satisfies
\begin{align}\label{cond:conservationofangularmomentum}
    J^i\eta_i = 0,
\end{align}
which guarantees the integral defining $\J$ only depends on the homology class of the sphere. For simply connected manifolds, \eqref{cond:conservationofangularmomentum} is equivalent to the existence of a twist potential \cite[Lemma 2.2]{dain2013lowerboundsareablack}. Lastly, the \emph{ADM mass} is given by the Lorentz length of the energy-momentum vector.
\begin{align}
    m = \sqrt{{E_\mathrm{ADM}^2-|P|^2}}
\end{align}
Note that the decay of $k$ in \eqref{cond:AF_asymptotics} forces linear momentum to vanish, so the ADM mass and energy coincide.

A heuristic argument by Penrose conjectures a sharp relationship between the mass enclosed by a black hole and the cross-sectional area $\A_e$ of its event horizon,
\begin{align}\label{conj:penrose}
    m^2 \geq \left(\sqrt{\frac{\A_e}{16\pi}} + \Q^2\sqrt{\frac{\pi}{\A_e}}\right)^2 + \frac{4\pi\J^2}{\A_e}.
\end{align}The conjecture asserts that equality is only satisfied when the initial data isometrically embeds as a hypersurface of the Kerr-Newman spacetime. In what follows, we shall assume that the area satisfies $\A\geq 4\pi\sqrt{\Q^4+4\J^2}$, since \eqref{conj:penrose} otherwise reduces to the positive mass bound $m^2 \geq \frac{1}{2}(\Q^2+\sqrt{\Q^4+4\J^2})$.

In the initial data set formulation of the Penrose inequality, the event horizon is represented by an apparent horizon in the distinguished end $\Mend$. Let $\Sigma\subset M$ be a closed, 2-sided hypersurface, and let $\nu$ denote the unit normal pointing toward $\Mend$. Its \emph{null expansions} are $\theta_\pm = H\pm\tr_\Sigma k$, where $H$ is the mean curvature of $\Sigma$ with respect to $\nu$. Equivalently, $\theta_\pm$ are the spacetime mean curvatures of $\Sigma$ in the null directions $\nu\pm n$, where $n$ is the future-directed timelike unit normal to the spacelike slice $(M,g,k)$. Thus $\theta_\pm$ measures the area variation of a light shell emitted outward from $\Sigma$ toward the future or past, respectively. A surface satisfying $\theta_+<0$ or $\theta_-<0$ is said to be \emph{outer} or \emph{inner trapped} and the occurrence of such surfaces is associated with strong gravitational fields. The limiting cases $\theta_+=0$ and $\theta_-=0$ define a \emph{marginally outer trapped surface} (MOTS) and a \emph{marginally inner trapped surface} (MITS), respectively. We shall refer to a union $\Sigma$ of disjoint MOTS and MITS components as an \emph{outermost apparent horizon} relative to $\Mend$ if no other apparent horizon in $\Mend$ encloses $\Sigma$. The existence of an outermost apparent horizon under the relevant trapped-region hypotheses follows from the construction of \cite{anderssonmetzger2009areahorizonstrappedregion}.

The Penrose inequality \eqref{conj:penrose} with $\Q=\J=0$ has been the subject of extensive study. In the time-symmetric setting ($k=0$), the inequality was proved for a connected black hole by \cite{huiskenilmanen2001inversemeancurvatureflow} via inverse mean curvature flow (IMCF). The multiple black hole setting was treated by \cite{bray1999proofriemannianpenroseconjecture} using a conformal flow and has since been generalized by \cite{braylee2009riemannianpenroseinequalitydimensions} for dimensions up to 7. Considerably less is known in the spacetime regime where $k$ is non-vanishing in general. Under spherical symmetry, the spacetime inequality was established by \cite{hayward1996gravitationalenergysphericalsymmetry,malecmurchadha1994trappedsurfacespenroseinequality} with rigidity obtained in \cite{bryden2021stabilityspacetimepositivemass,bray2009jangequationapproachpenrose}. The inequality has also been shown in the cohomogeneity one setting by \cite{khuri2024spacetimepenroseinequalitycohomogeneity} and for 2-convex initial data sets via $\mathbf{P}$-IMCF by \cite{dong2026penroseconjectureinitialdata}, the former result extending to higher dimensions. Recently, \cite{allen2025proofspacetimepenroseinequality} showed that there is a universal constant $C\ll 1$ such that $m\geq C\sqrt{\A}$. Their argument combines the Jang deformation on initial data sets with the harmonic level set techniques of \cite{bray2019harmonicfunctionsmass3dimensional} and the Poincar\'e-Faber-Szeg\H{o} inequality from \cite{dong2025stabilityriemannianpenroseinequality}.

In the charged setting, the time-symmetric case with multiple black hole components was confirmed by \cite{khuri2014riemannianpenroseinequalitycharge} using conformal flow. Assuming axisymmetry and maximality ($\tr_g k = 0$), \cite{anglada2018} obtains a Penrose-type inequality with angular momentum provided that a smooth IMCF exists starting from a connected minimal apparent horizon. Also in the maximal setting, \cite{dain2013lowerboundsareablack} derives a partial Penrose inequality for a connected apparent horizon. A different variational approach appears in \cite{khuri2019penrosetypeinequalityangularmomentum}, where a charged Penrose-type inequality with angular momentum is obtained by exhibiting the Kerr-Newman black hole as a minimizer for a suitable reduced harmonic energy. Recently, \cite{dolmen2026penroseinequalitycharge2convex} extended the work of \cite{dong2026penroseconjectureinitialdata} to prove the charged Penrose inequality for 2-convex initial data sets by modifying the monotone Hawking-type quantity to include electromagnetic charge.

The goal of the present article is to show that \eqref{conj:penrose} holds with a universal suboptimal constant, thereby generalizing the results of \cite{allen2025proofspacetimepenroseinequality} to the setting of charge and angular momentum.

\begin{theorem}\label{thm:main_AF}
    Let $(M,g,k,\E)$ be a complete, charged, axisymmetric, and asymptotically flat initial data set satisfying the charged dominant energy condition. Suppose the axial Killing field $\eta$ satisfies the twist-free condition $\eta^\flat\wedge d\eta^\flat=0$. Then
    \begin{align}\label{eqn:main_AF}
        m \geq \C\ls\lp\sqrt{\frac{\A}{16\pi}}+\Q^2\sqrt{\frac{\pi}{\A}}\rp^2 + \frac{\mathcal{J}^2}{\mathscr{R}_\eta^2} \rs^{1/2}
    \end{align}
    holds with $\C=3.3701 \times10^{-4}$, where the following quantities are defined with respect to an end $\Mend$:
    \begin{itemize}
        \item $m,\Q,\J$ are respectively the ADM mass, total charge, and ADM angular momentum,
        \item $\A$ is the minimal area required to enclose an outermost apparent horizon,
        \item $\mathscr{R}_\eta^2=\mathcal{R}_\eta(S_{R_*})^3/\sqrt{\A}$, where $R_\eta$ is defined in \eqref{def:radius} and $R_*>0$ is given by \Cref{prop:masschargeangularmomentum}.
    \end{itemize}
\end{theorem}

\begin{remark}
    This improves the inequality $m\geq\C\smash{\sqrt{\A/(16\pi)}}$ obtained by \cite{allen2025proofspacetimepenroseinequality}, where $\C=1.7502\times10^{-8}$. The primary sources of this improvement are elementary refinements to the estimates of \Cref{sec:area}. An additional improvement is found by replacing the capacity-volume argument with a Poincar\'e-Szeg\H{o} inequality for function rearrangements.
\end{remark}

\begin{remark}
    The axisymmetry and twist-free assumptions may be omitted to obtain the inequality with $\J=0$.
\end{remark}

The inequality \eqref{conj:penrose} also naturally arises in the setting of asymptotically hyperboloidal initial data sets. Geometrically, such initial data describe spacelike hypersurfaces that asymptote to the null cone at infinity and thus intersect null infinity along the ends. The asymptotic model in this setting is hyperbolic $3$-space $(\HS,b)$, realized as the hyperboloid $t=\sqrt{1+|x|^2}$ in Minkowski space. In polar coordinates, $\HS=[0,\infty)\times S^2$ and $b=(1+r^2)^{-1}dr^2 + r^2\sigma$, where $\sigma$ is the unit round metric on the 2-sphere. We shall say an initial data set is \emph{asymptotically hyperboloidal} if there is a compact set $K$ such that $\smash{M\setminus K=\bigsqcup_{\ell=1}^{\ell_0} M^\ell_\text{end}}$ where each end $\Mend^\ell$ is diffeomorphic to the complement of a hyperbolic 3-ball and admit coordinates such that
\begin{align}
    \begin{gathered}\label{cond:AH_asymptotics}
        g = b +\frac{\mathbf{m}}{r}+O_6(r^{-2}), \qquad 
        k = b + \frac{\mathbf{p}}{r} + O_4(r^{-2}), \\ 
        \mu_E,J = O_3(r^{-3-q_0}), \qquad 
        \E = O_1(r^{-2}),
    \end{gathered}
\end{align}
where $q_0>0$, $\mathbf{m}$ and $\mathbf{p}$ are symmetric 2-tensors on $S^2$, and $O_l(r^{-q})$ represents a tensor in the weighted space $C^l_{-q}(\HS)$. The expansions in \eqref{cond:AH_asymptotics} are known as \emph{Wang asymptotics} \cite{wang2001massasymptoticallyhyperbolicmanifolds}. As stated, the order of decay and differentiability is stronger than those typically imposed, but are sufficient for the Jang equation analysis of \cite{sakovich2021jangequationpositivemass}. Assuming these asymptotic, the \emph{mass-aspect function} $\tr_\sigma(\mathbf{m}+2\mathbf{p})$ gives rise to a well-defined \emph{total energy} and \emph{linear momentum}, given by
\begin{align}\label{def:hyperbolicenergy}
    E_\text{hyp} = \frac{1}{16\pi} \int_{S^2} \tr_\sigma(\mathbf{m}+2\mathbf{p}) \, dA_\sigma, \qquad
    (P_\mathrm{hyp})_i = \frac{1}{16\pi} \int_{S^2} x^i\tr_\sigma(\mathbf{m}+2\mathbf{p}) \, dA_\sigma,
\end{align}
where $x^1,x^2,x^3$ are the coordinate functions on $\R^3$ restricted to $S^2$. The \emph{total mass} is the Minkowskian length of the energy-momentum vector. 
\begin{align}
    m=\sqrt{E_\mathrm{hyp}^2-|P_\mathrm{hyp}|^2}
\end{align}
The total charge and ADM angular momentum are given as before. By building on the suboptimal-constant framework of \cite{allen2025proofspacetimepenroseinequality}, we are able to establish a corresponding spacetime hyperboloidal Penrose inequality with charge and angular momentum.

\begin{theorem}\label{thm:main_AH}
    Let $(M,g,k,\E)$ be a complete, charged, axisymmetric, and asymptotically hyperboloidal initial data set satisfying the charged dominant energy condition. Suppose the axial Killing field $\eta$ satisfies the twist-free condition $\eta^\flat\wedge d\eta^\flat=0$. Then
    \begin{align}
        E_\text{hyp} \geq \C\ls\lp\sqrt{\frac{\A}{16\pi}}+\Q^2\sqrt{\frac{\pi}{\A}}\rp^2 + \frac{\mathcal{J}^2}{\mathscr{R}_\eta^2} \rs^{1/2}
    \end{align}
    holds with $\C = 3.3701 \times10^{-4}$, where $E_\text{hyp}$ is the total energy of an end. All other quantities are as defined in \Cref{thm:main_AF}.
\end{theorem}

\subsection*{Acknowledgements}
The author would like to thank Marcus Khuri for valuable discussions, as well as Santiago Cordero Misteli and Ted Roe for helpful comments. This paper was partially supported by NSF
Grant DMS-2405045 and Simons Foundation International, Ltd.

\section{Charged Jang reduction and a mass-charge inequality}\label{sec:mass-charge}

The goal of this section is to establish a lower bound for the mass-charge gap $m-|Q|$. Following the approach of \cite{allen2025proofspacetimepenroseinequality}, we construct an auxiliary initial data set and solve the Jang equation to obtain a manifold of weakly positive scalar curvature. The new step is to lift the electric field to the Jang graph using \cite[Equation 2.7]{disconzi2012penroseinequalitychargedblack} and to solve the charged harmonic equation on the resulting charged manifold. We then establish a Lichnerowicz-type integral inequality as in \cite[Theorem 3.4]{bray2021spacetimeharmonicfunctionsapplications} for charged harmonic equations with linear growth.

\subsection{Charged Jang deformation}\label{subsec:Jang}

Roughly speaking, a generalized exterior region represents the region outside of an outermost apparent horizon in $M$. The notion was first introduced by \cite{hirsch2021spacetimeharmonicfunctionsmass} for asymptotically flat initial data sets and later extended to the asymptotically hyperboloidal setting by \cite{allen2025proofspacetimepenroseinequality}, whose result is reproduced below. The construction involves removing an appropriate collection of MOTS/MITS and iteratively passing to a finite cover to reduce the Betti number.

\begin{lemma}[Existence of generalized exterior regions, \cite{allen2025proofspacetimepenroseinequality}]\label{lem:generalizedexteriorregion}
    Let $(M,g,k)$ be an initial data set which is either asymptotically flat or asymptotically hyperboloidal and satisfies the dominant energy condition. Then for each end $\Mend$, there exists an associated outermost apparent horizon $\Sigma\subset M$ and an initial data set $(\ext{M},\ext{g},\ext{k})$ with boundary satisfying the following properties.
    \begin{enumerate}[label=(\arabic*)]
        \item The second relative homology is trivial, $H_2(\ext{M},\del\ext{M};\Z)=0$.
        \item $\ext{M}$ has a single end which is isometric as initial data to $(\Mend,g,k)$.
        \item $\del\ext{M}$ is an outermost apparent horizon.
        \item The minimal enclosing area of $\del\ext{M}$ is not exceeded by the minimal enclosing area of $\Sigma$.
        \begin{align*}
            \A_\text{min}(\Sigma)\leq \A_\text{min}(\del\ext{M})
        \end{align*}
    \end{enumerate}
\end{lemma}

The \emph{Jang surface} $(\Mj,\gj)$ is the graph of a smooth function $f:\ext{M}\to\R$ solving the \emph{Jang equation}
\begin{align}\label{eqn:jang}
    \lp \ext{g}^{ij} - \frac{f^i f^j}{1+|\grad f|^2} \rp \lp \frac{\grad_{ij} f}{\sqrt{1+|\grad f|^2}} - (\ext{k})_{ij} \rp = 0.
\end{align}
If the initial data set is axisymmetric, the solution to \eqref{eqn:jang} shall also be required to be axisymmetric, $\eta(f)=0$.

In the asymptotically flat setting, \cite[Theorem 3.1]{metzger2009blowupjangequationoutermost} shows that the Jang equation admits a solution with zero Dirichlet boundary along the end and prescribed blow-up at an outermost apparent horizon. In particular, the solution is shown to satisfy
\begin{alignat}{2}
    &f=O_3(r^{1-2q}) &\quad &\text{as} \quad r\to\infty, \label{cond:jangasymptoticsAF} \\
    &f(p)\to\pm\infty &&\text{as} \quad \text{dist}(p,\del_\pm\ext{M}) \to 0, \label{cond:jangblowupAF}
\end{alignat}
where $r$ is the radial coordinate on the asymptotically flat end and $\del\ext{M}=\del_+\ext{M}\sqcup\del_-\ext{M}$ consists of MOTS/MITS components respectively. Observe that the regularization procedure used in \cite{metzger2009blowupjangequationoutermost} preserves axisymmetry. Hence if the initial data is axisymmetric, then $f$ may also be taken to be axisymmetric=. 

Similarly, in the asymptotically hyperboloidal setting there exists a solution with
\begin{alignat}{2}
    &f=\sqrt{1+r^2} + \alpha \log{r} + \psi + O_3(r^{-1+\varepsilon}) &\quad &\text{as} \quad r\to\infty, \label{cond:jangasymptoticsAH} \\
    &f(p)\to\pm\infty &&\text{as} \quad \text{dist}(p,\del_\pm\ext{M}) \to 0, \label{cond:jangblowupAH}
\end{alignat}
for some $\varepsilon\in(0,1)$ and $\alpha,\psi\in C^3(S^2)$, where $r$ is the radial coordinate on the asymptotically hyperboloidal end \cite[Proposition 3.1]{allen2025proofspacetimepenroseinequality}. The proof of the asymptotic expansion depends heavily on the barrier estimates carried out in \cite[Section 3]{sakovich2021jangequationpositivemass}. Again, the solution obtained preserves the axisymmetry of the initial data.

A refinement of the blow-up in \cref{cond:jangblowupAF,cond:jangblowupAH} produces asymptotically cylindrical ends $\{\End_l\}_{l=1}^{l_0}$ for each component of $\del\ext{M}$ \cite[Theorem 4.1]{metzger2009blowupjangequationoutermost}. Let $\widecheck g=dt^2+g_0$ denote the product metric on the cylinder $\mathscr{C}^\pm=(0,\pm\infty)\times\del\ext{M}$. Then the Jang surface is a graph over $\mathscr{C}^\pm$ and the pullback of $\gj$ satisfies
\begin{align}\label{cond:cylinder_asymptotics}
    |\widecheck \grad^a(\gj-\widecheck g)|_{\widecheck g} = o(1) \quad \text{as} \quad t\to\pm\infty \quad \text{in} \quad \mathscr{C}^\pm,
\end{align}
for all $a=0,1,2,\ldots$, where $\widecheck\grad$ denotes covariant differentiation with respect to $\widecheck{g}$. The scalar curvature of the Jang graph satisfies
\begin{align}\label{eqn:jangscalarcurvature}
    R_{\gj} = 16\pi (\mu_E-J(w)) + 2|\E|_g^2 + |h-k|_{\gj}^2 + 2|X|_{\gj}^2 - 2\Div_{\gj}X,
\end{align}
where 
\begin{align}
    w = \frac{\grad f}{\sqrt{1+|\grad f|^2}}, \quad
    h = \frac{\grad^2 f}{\sqrt{1+|\grad f|^2}}, \quad
    X = (h-k)(w,\cdot).
\end{align}
The remaining properties of the Jang surface are reproduced below from \cite[Proposition 3.1]{allen2025proofspacetimepenroseinequality} for convenience. Their proof extends to the charged setting without modification, since the Einstein–Maxwell constraints and charged dominant energy condition only enter in the subsequent scalar-curvature estimates on the Jang graph.

\begin{proposition}[Existence of the Jang graph, \cite{allen2025proofspacetimepenroseinequality}]\label{prop:jang}
    Let $(M,g,k,\E)$ be an (axisymmetric) initial data set for the Einstein-Maxwell equations which is either asymptotically flat or asymptotically hyperboloidal and satisfies the charged dominant energy condition. Then for each end $\Mend$, there exists a Jang triple $(\Mj,\gj,X)$ consisting of a complete (axisymmetric) Riemannian manifold with asymptotic ends $\End_0 \bigsqcup_{l=1}^{l_0} \End_l$ and a smooth vector field $X$ satisfying the properties in \Cref{table:jang}.
\end{proposition}

\begin{table}[ht]
    \centering
    \begin{tabular}{|c||c|c|c|}
        \hline
        $\mathbf{(M,g)}$ & \textbf{Asymptotically flat} & \textbf{Asymptotically hyperboloidal} & \textbf{\#} \\ \hline
        \multirow{4}{*}{$(\Mj,\gj)$} & \multicolumn{2}{c|}{Trivial second homology: $H_2(\Mj,\bigcup_{l=1}^{l_0}\End_l;\Z)=0$} & \multicolumn{1}{c|}{(1)} \\ \cdashline{2-4}[1pt/2pt]
        & \multicolumn{2}{p{31em}|}{If $\jang\Sigma\subset \Mj$ is a surface separating $\End_0$ from $\bigcup_{l=1}^{l_0}\End_l$, then $|\jang \Sigma|\geq\A_\text{min}(\Sigma)$,} & \multicolumn{1}{c|}{\multirow{2}{*}{(2)}} \\[-6pt] 
        & \multicolumn{2}{p{31em}|}{where $\Sigma \subset M$ is the outermost apparent horizon given by \Cref{lem:generalizedexteriorregion}.} & \multicolumn{1}{c|}{} \\ \cdashline{2-4}[1pt/2pt]
        & $R_{\gj}-2|\E|_g^2\in L^1(\End_0)$ & $R_{\gj}-2|\E|_g^2 + 2\Div_{\gj}X\in L^1(\End_0)$ & \multicolumn{1}{c|}{(3)} \\ \hline
        \multirow{2}{*}{$\End_0$} & AF of order $\qj=q$ & \multicolumn{1}{p{18em}|}{AF of order $\qj=1$ with the $(\mu_E,J)$ equa-} & \multicolumn{1}{c|}{\multirow{2}{*}{(4)}} \\[-6pt]
        & & \multicolumn{1}{p{18em}|}{tion replaced by $R_{\gj}-2|\E|_g^2=O(r^{-3})$} & \multicolumn{1}{c|}{} \\ \cline{2-4}
        & $E_\mathrm{ADM}(\End_0)=E_\mathrm{ADM}(\Mend)$ & $E_\mathrm{ADM}(\End_0)=2E_\text{hyp}(\Mend)$ & \multicolumn{1}{c|}{(5)} \\ \hline 
        \multirow{3}{*}{$X$} & \multirow{2}{*}{$|X|_{\gj} = O(r^{-1-2q})$ along $\End_0$} & $|X|_{\gj}=O(r^{-2})$ along $\End_0$ & \multicolumn{1}{c|}{\multirow{2}{*}{(6)}} \\[-2pt] 
        & & $\lim_{r\to\infty} \int_{S_r} \langle X,\nu \rangle_{\gj} \, d\Aj = -4\pi E_\mathrm{ADM}(\End_0)$ & \multicolumn{1}{c|}{} \\ \cdashline{2-4}[1pt/2pt]
        & \multicolumn{2}{c|}{$|X|_{\gj}$ is bounded on $\End_l$ for $l=1,\ldots,l_0$} & \multicolumn{1}{c|}{(7)} \\ \hline
    \end{tabular}
    \caption{}
    \label{table:jang}
\end{table}

Let the \emph{Jang electric field} $\Ej$ be the smooth vector field on $\Mj$ given by
\begin{align}\label{def:jangelectricfield}
    \Ej = \frac{\E + \langle\E,\grad f\rangle_g \grad f}{\sqrt{1+|\grad f|_g^2}}.
\end{align}
This definition was introduced by \cite{disconzi2012penroseinequalitychargedblack} to obtain a charged Penrose inequality, provided a solution exists to the coupled Jang-IMCF system. The key properties of $\Ej$ are stated below. Note that their proof does not depend on $f$ or the asymptotic conditions of the underlying manifold. Altogether, we shall refer to $(\Mj,\gj,\Ej)$ as the \emph{charged Jang surface}.

\begin{proposition}[\cite{disconzi2012penroseinequalitychargedblack}]\label{prop:jangelectricfield}
    Let $(M,g,\E)$ be a charged manifold. Then for any smooth $f$, the electric field $\Ej$ defined by \eqref{def:jangelectricfield} satisfies the following properties:
    \begin{align}
        |\Ej|_{\gj}\leq|\E|_g, \qquad
        \Div_{\gj}\jang\E = \dfrac{\Div_g \E}{\sqrt{1+|\grad f|^2_g}}, \qquad
        \Qj = \Q,
    \end{align}
    where $\Qj$ is the total charge associated to $(\Mj,\gj,\Ej)$.
\end{proposition}
\subsection{Charged harmonic functions}

Let $(M,g,\E)$ be a complete, charged, asymptotically flat manifold. We are interested in solutions of the equation
\begin{align}\label{eqn:chargedharmonic}
    \begin{cases}
        \Delta u - \langle \grad u,\E\rangle = 0 & \text{on}\quad M, \\
        u=a_ix^i + O_2(|x|^{1-q}) & \text{as}\quad |x|\to\infty,
    \end{cases}
\end{align}
where $a_i$ are fixed constants and $x^i$, $q$ are respectively the coordinates and order of the asymptotically flat end. We shall refer to $\Deltac u = \Delta u - \langle\grad u,\E\rangle$ as the \emph{charged Laplacian} of $u$.

There is a special case where the charged Laplacian arises as the Laplace-Beltrami operator of a metric. Suppose $\E$ admits a scalar potential $\phi$ and consider the conformal metric $\gc=e^{-2\phi}g$. Then $\Deltac u = e^{-2\phi}\Delta_{\gc} u$, where the Laplacian is taken with respect to $\gc$. In general, there are both local and global obstructions to the existence of an electric potential--namely that $\E$ may fail to be closed, or if closed, $[\E^\flat]\in H^1(M)$ may be nontrivial.

Even without an electric potential, it is still possible to construct a symmetric 2-tensor whose trace is equal to the charged Laplacian. Define an affine connection $\grad^\E$ on $(M,g)$ by
\begin{align}\label{def:chargedconnection}
    \grad^\E_V W = \grad_V W - \langle\E,V\rangle W-\langle\E,W\rangle V+\langle V,W\rangle \E,
\end{align}
where $V,W$ are smooth vector fields on $M$. It is straightforward to check that $\grad^\E$ is torsion-free and satisfies
\begin{align}\label{eqn:metricdeficit}
    \grad^\E g = 2\E^\flat\otimes g.
\end{align}
The existence of an electric potential is therefore equivalent to the existence of a conformal metric $g_\phi = e^{-2\phi}g$ such that the connection $\grad^\E$ is $g_\phi$-metric compatible, since \eqref{eqn:metricdeficit} implies
\begin{align}
    \grad^\E g_\phi 
    = 2(\E^\flat-d\phi)\otimes g_\phi
    = 0
    \iff \E^\flat = d\phi.
\end{align}
Hence $\grad^\E$ defines a Weyl connection on the conformal class of $(M,g)$. Define the \emph{charged Hessian} of $u$ to be the symmetric 2-form $\gradc^2 u$ given by
\begin{align}\label{def:chargedhessian}
    \gradc^2 u = \grad^\E du.
\end{align}
In coordinates, $\gradc_{ij} u = \grad_{ij} u + 2\E_{(i} u_{j)} - \langle\E,\grad u\rangle g_{ij}$, where $u_i$ denotes the partial derivative of $u$. Next, we prove a Kato-type inequality for charged harmonic functions (cf. \cite[Equation~3.11]{hirsch2021spacetimeharmonicfunctionsmass}).

\begin{lemma}\label{lem:kato}
    Let $u$ be a charged harmonic function on $(M,g,\E)$. Then
    \begin{align}\label{lem:kato:eqn1}
        \Big\vert \grad|\grad u|+|\grad u|\E \Big\vert^2 \leq \frac{2}{3} |\gradc^2 u|^2,
    \end{align}
    almost everywhere, where $\gradc^2 u$ is defined by \eqref{def:chargedhessian}. If $\Div_g\E=0$, then \eqref{lem:kato:eqn1} is equivalent to
    \begin{align}\label{lem:kato:eqn2}
        |\grad|\grad u||^2 + 2|\grad u|\Div_g(|\grad u|\E) + |\grad u|^2|\E|^2
        \leq \frac{2}{3}|\gradc^2 u|^2.
    \end{align}
\end{lemma}
\begin{proof}
    On the set where $|\grad u|$ is differentiable, $|\grad|\grad u||=0$ whenever $|\grad u|$ achieves its minimum. Consequently, whenever $|\grad u|=0$, the inequality is trivially satisfied. Since the set where $|\grad u|$ vanishes and is non-differentiable has measure zero, we may assume for the remainder of the proof that $|\grad u|\neq 0$. We compute
    \begin{equation}
        \begin{aligned}
            \gradc^2 u(\grad u,\cdot)^\sharp
            &= \grad^2 u(\grad u,\cdot)^\sharp + \E(u) \grad u + |\grad u|^2 \E - \E(u) \grad u \\
            &= |\grad u|\grad|\grad u| + |\grad u|^2 \E.
        \end{aligned}
    \end{equation}
    Dividing by $|\grad u|$ yields $\gradc^2u(\frac{\grad u}{|\grad u|},\cdot)^\sharp = \grad|\grad u|+|\grad u|\E$. Because $\gradc^2 u$ is trace-free by assumption, we may apply \Cref{lem:tracelessestimate} to obtain \eqref{lem:kato:eqn1}. If $\E$ is divergence-free, then
    \begin{align}
        \Div_g(|\grad u|\E) 
        = \langle\grad|\grad u|,\E\rangle + |\grad u|\Div_g\E
        = \langle\grad|\grad u|,\E\rangle,
    \end{align}
    so expanding the left-hand side of \eqref{lem:kato:eqn1} yields \eqref{lem:kato:eqn2}.
\end{proof}

\subsubsection{Existence and asymptotic behavior}\label{sec:chargedharmonicexistence}

Let $(\Mj,\gj,\Ej)$ be the charged Jang surface from \Cref{subsec:Jang} and consider the asymptotically cylindrical ends $\{\End_l\}_{l=1}^{l_0}$ as isometrically embedded in $(\ext{M}\times\R,g+dt^2)$. Denote by $\End_l^t$ the cross-section obtained by intersecting such an end with the $t$-level set. There exists a level $t_0>0$ such that $\End_l^t$ is either empty or consists of a single sphere for all $|t|\geq t_0$ and $l=1,\ldots,l_0$. Let $\Mj_t$ be the metric completion of the component of $\Mj \,\setminus\, \{\End_l^{\pm t}\}_{l=1}^{l_0}$ containing the asymptotically flat end $\End_0$, where $|t|\geq t_0$. By \eqref{cond:cylinder_asymptotics}, the truncated regions $\Mj_t$ may be capped off by 3-balls $\{\til B^t_l\}_{l=1}^{l_0}$ such that the resulting geometry is bounded uniformly and independently of $t$. Denote the subsequent capped manifold by $(\til M_t,\gt_t)$. Observe that $\til M_t$ is complete with one asymptotically flat end and satisfies $\gt_t=\gj$ on $\Mj_t$. 

\begin{figure}[ht]
    \centering
    \begin{subfigure}[t]{0.52\textwidth}
        \centering
        \includegraphics[height=4.5cm]{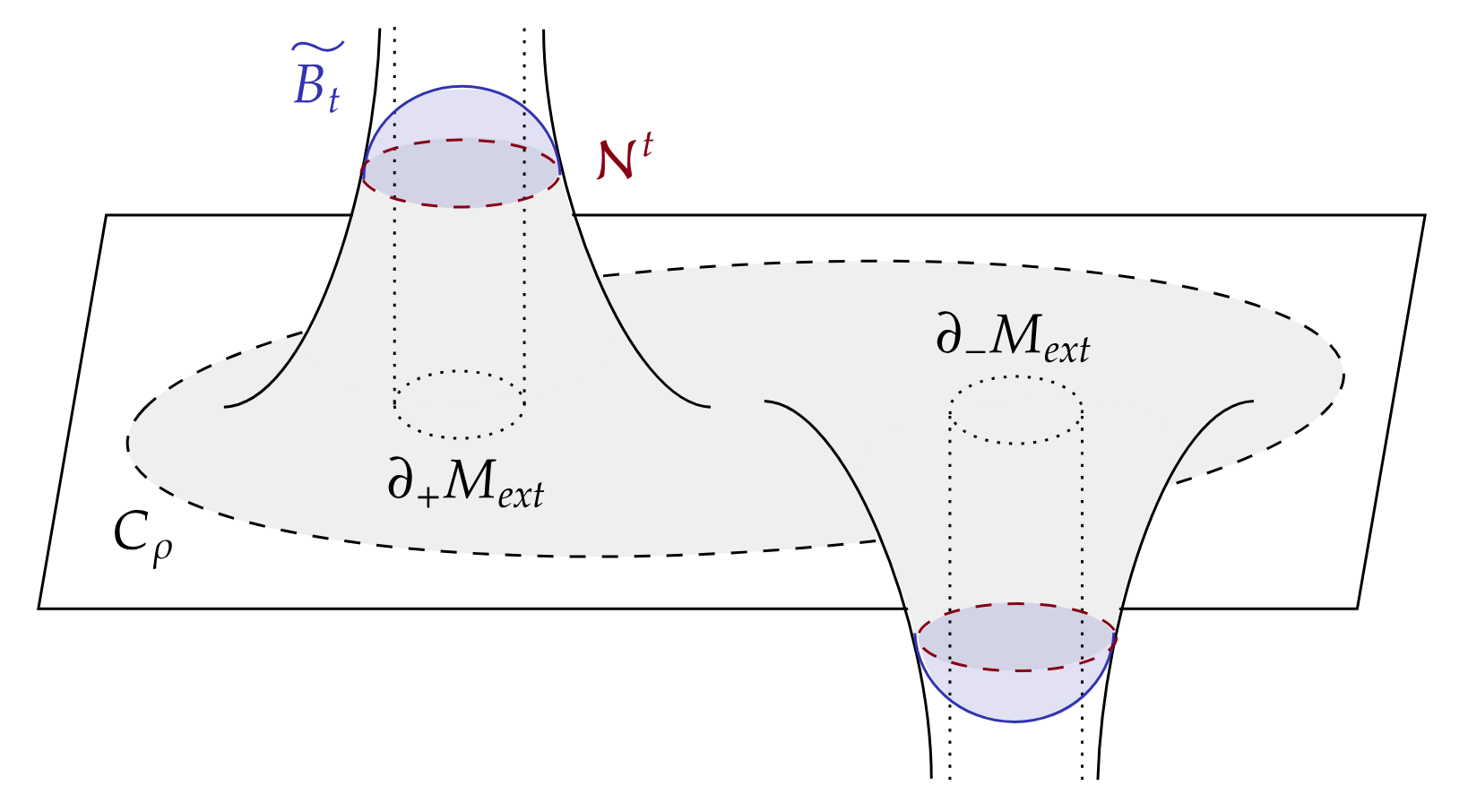}
        \caption{Capped domain $\til M_t$ and a coordinate cylinder $C_\rho\subset\End_0$.}
    \end{subfigure}
    \hfill
    \begin{subfigure}[t]{0.47\textwidth}
        \centering
        \includegraphics[height=4.5cm]{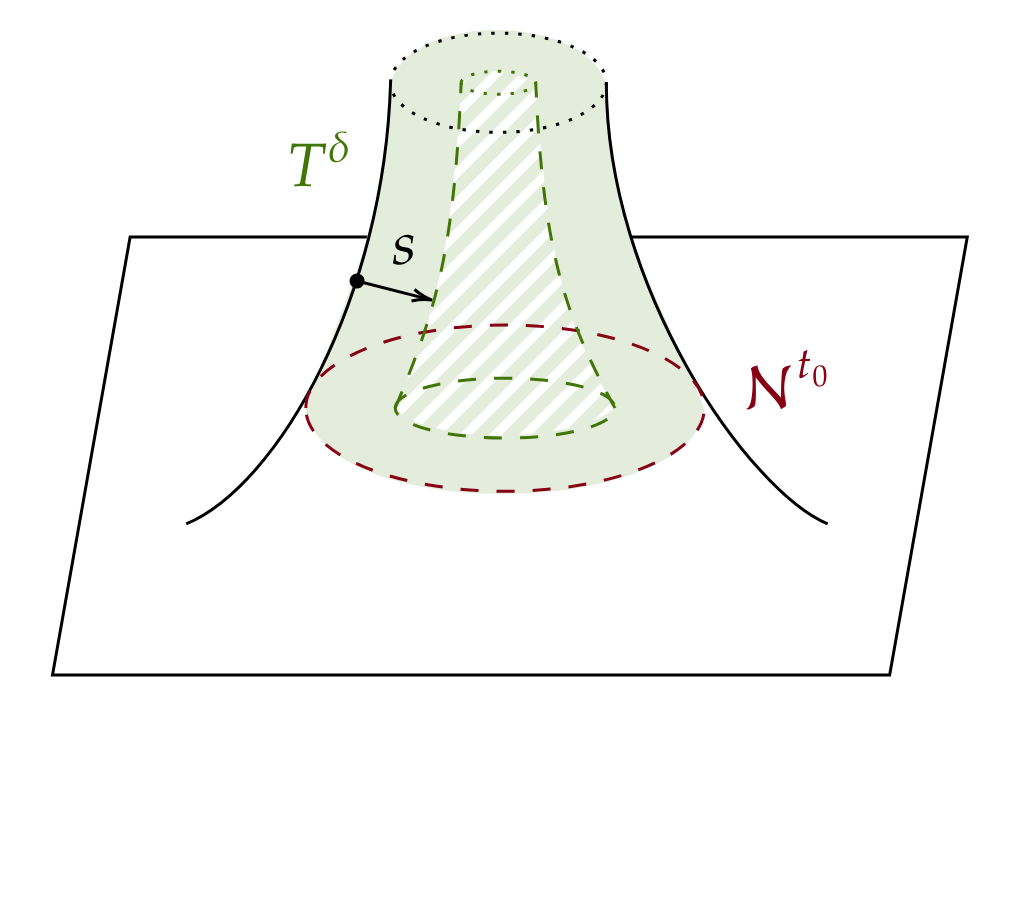}
        \caption{Half-tubular neighborhood of $\End$.}
    \end{subfigure}
\end{figure}

Next, we extend the electric field to the capped domain. Let $n$ be the unit normal to $\End_l\subset\ext{M}\times\R$, with the orientation chosen so that $n\to(\nu_\text{out},0)$ as $|t|\to\infty$, where $\nu_\text{out}$ is the outward unit normal to $\del_\pm\ext{M}\subset\ext{M}$. Let $d_{\End_l}:\ext{M}\times\R\to\R$ be the signed distance to $\End_l$, with the sign chosen so that $d_{\End_l}>0$ in the prescribed outward direction. For $\delta>0$ sufficiently small, define the half-tubular neighborhood
\begin{align}
    T^\delta_l = \{ \exp_p(sn_p) :\, p\in\End_l, \, |t(p)|\geq t_0, \, 0\leq s < \delta\},
\end{align}
after increasing $t_0$ if necessary so that the exponential map gives tubular coordinates. Let $\eta$ be a smooth bump function supported on $[-\delta,\delta]$ with $\eta(0)=1$. We extend $\Ej$ to $T^\delta_l$ by setting 
\begin{align}
    \Et(\exp_p(sN_p)) = \eta(s)\Ej(p), \quad p\in\End_l, \quad 0\leq s < \delta.
\end{align}
Then $\Et=\Ej$ on $\End_l$ and extends smoothly by zero across the outer boundary of $T_l^\delta$. The extended electric field need not satisfy the charged dominant energy condition on the caps, but this does not affect the estimates used below. For $|t|\geq t_0$, we may apply \cite[Theorem 3.1]{bartnik1986massasymptoticallyflatmanifold} to obtain solutions to the charged harmonic equation on the capped Jang graph.
\begin{align}\label{eqn:cappedchargedharmonic}
    \begin{cases}
        \Delta_{\gt_t} u_t - \langle \Et,\gradt u_t\rangle_{\gt_t} = 0 & \text{on}\quad\til M_t, \\
        u_t = a_i x^i + O_2(r^{1-\qj}) & \text{as}\quad r\to\infty
    \end{cases}
\end{align}
Moreover, we obtain the same decay along the cylindrical ends as in \cite{allen2025proofspacetimepenroseinequality}.

\begin{lemma}\label{lem:chargedharmonic}
    There exists a charged harmonic function $u$ with the asymptotics of \eqref{eqn:cappedchargedharmonic} on $\End_0$ and a sequence of charged harmonic functions $u_{t_n}$ on $\smash{\til M_{t_n}}$ for some $t_n\to\infty$ such that the following holds.
    \begin{enumerate}[label=(\arabic*)]
        \item $u_{t_n}\to u$ in $C^{2,\alpha}_\text{loc}(\Mj)$ for any $\alpha\in(0,1)$.
        \item $|\gradt u_{t_n}|\to 0$ uniformly on the caps $\{\til B_l\}_{l=1}^{l_0}$.
        \item For each cylindrical end $\End_l$, there exists a constant $c_l$ such that
        \begin{align*}
            |\gradj^a (u-c_l)| = o(1) \quad \text{as} \quad |t|\to\infty,
        \end{align*}
        for $a=0,1,2$.
    \end{enumerate}
\end{lemma}
\begin{proof}
    To show (1), we may apply a standard barrier argument to obtain uniform pointwise bounds on compact subsets for solutions of \eqref{eqn:cappedchargedharmonic}. Then elliptic theory implies H\"older estimates of the form
    \begin{equation}\label{eqn:uniformholder}
        |u_t|_{C^{k,\alpha}(\Omega_t)} \leq C_k,
    \end{equation}
    for any precompact $\Omega_t\subset\til M_t$ of uniformly bounded geometry, where $C_k$ are constants independent of $t$. By a diagonal argument, there is a subsequence $u_{t_n}$ converging to $u$ in $C_\mathrm{loc}^{2,\alpha}(\Mj)$ for any $\alpha\in(0,1)$, where $u$ is a charged harmonic function on $(\Mj,\gj)$ admitting the asymptotics of \eqref{eqn:cappedchargedharmonic} on $\End_0$ (refer to \cite[Lemma 3.3]{allen2025proofspacetimepenroseinequality} for the full details).

    The estimates for $u_t$ along the cylindrical ends established in \cite{allen2025proofspacetimepenroseinequality} extend to the charged setting after accounting for the new terms introduced by the electric field. Observe that the divergence-free property of $\E$ is inherited by the Jang electric field according to \Cref{prop:jangelectricfield}. Additionally, since $\E$ is smooth and $\del\ext{M}$ is compact, there is a compact neighborhood $\Omega$ of $\del\ext{M}$ such that $|\E|_{C^k(\Omega)}<\infty$ for every $k\geq 0$. In what follows, we shall denote $u_n=u_{t_n}$ for convenience. Suppose $t_0\leq t_m<t_n$ and let $\End_l^{m,n} \subset \End_l$ be the annular region bounded between the cross-sections $\End^{\pm t_m}_l$ and $\End^{\pm t_n}_l$. After multiplying the charged Laplace equation by $u_n$, we may integrate by parts to find
    \begin{align}
        \norm{\gradj u_n}^2_{L^2(\End_l^{m,n})}
        &= \int_{\del\End_l^{m,n}} u_n \gradj_\nu u_n \, d\Aj + \int_{\End_l^{m,n}} u_n \langle\Ej,\gradj u_n\rangle_{\gj} \, d\Vj,
    \end{align}
    where $\nu$ is the unit outer normal to the boundary. Let $\pi:\Mj\to M$ denote the projection from the graph of $f$ to $M$. Since $\E$ is divergence free, by \Cref{prop:jangelectricfield} we have
    \begin{equation}
        \begin{aligned}
            \int_{\End_l^{m,n}} u_n \langle\Ej,\gradj u_n\rangle_{\gj} \, d\Vj
            &= \frac{1}{2} \int_{\End_l^{m,n}} \Div_{\gj}(u_n^2\Ej) - u_n^2 \Div_{\gj}\Ej \, d\Vj \\
            &= \frac{1}{2} \int_{\del\End_l^{m,n}} u_n^2 \langle\Ej,\nu\rangle_{\gj} \, d\Aj \\
            &\leq \frac{1}{2} |\E|_{C^0(\pi(\del\End_l^{m,n}))} \int_{\del\End_l^{m,n}} u_n^2 \, d\Aj.
        \end{aligned}
    \end{equation}
    Then by \eqref{eqn:uniformholder} and \eqref{cond:cylinder_asymptotics}, there exists a constant $C$ independent of $m,n$ such that
    \begin{align}\label{eqn:uniformH1}
        \norm{\gradj u_n}^2_{L^2(\End_l^{m,n})}
        \leq C |\del\End_l^{m,n}|
        \leq C.
    \end{align}
    For the second derivative, we may square the charged Laplace equation and integrate by parts to get
    \begin{align}
        \big\lVert\gradj^2 u_n\big\rVert^2_{L^2(\End_l^{m,n})}
        &= \int_{\del\End_l^{m,n}} \gradj^2 u_n(\gradj u_n,\nu) - \langle\Ej,\gradj u_n\rangle_{\gj} \gradj_\nu u_n \, d\Aj
        + \int_{\End_l^{m,n}} \langle\Ej, \gradj u_n\rangle_{\gj}^2 - \Ric_{\gj}(\gradj u_n,\gradj u_n) \, d\Vj \nonumber \\
        &\leq \int_{\del\End_l^{m,n}} |\gradj^2 u_n|_{\gj} |\gradj u_n|_{\gj} + |\E|_g |\gradj u_n|_{\gj}^2 \, d\Aj + \int_{\End_l^{m,n}} \lp |\E|_g^2 + |\Ric_{\gj}| \rp |\gradj u_n|^2 \, d\Vj \nonumber \\
        &\leq C \label{eqn:uniformH2},
    \end{align}
    where we again used \Cref{prop:jangelectricfield}, \eqref{cond:cylinder_asymptotics}, \eqref{eqn:uniformholder}, as well as \eqref{eqn:uniformH1} to obtain a uniform constant $C$. By performing similar but lengthy computations, we may bootstrap to obtain the same control over all higher order derivatives. Then by Fatou's lemma, the same estimate holds in the limit $n\to\infty$.
    \begin{equation}\label{eqn:uniformsobolev}
        \sum_{a=1}^4 \big\lVert\gradj^a u\big\rVert_{L^2(\End_l)} \leq C
    \end{equation}
    Lastly, apply the Sobolev embedding $H^3\hookrightarrow C^1$ on $\End_l^{(t-1),t}$ to obtain $|\gradj^a u|=o(1)$ as $|t|\to\infty$ for $a=1,2$.

    To show (3), take a sequence $t_n\to\infty$ and consider the annular regions $\End_l^{(t_n-1),t_n}$. Then by the Poincar\'e inequality, there exists a constant $C$ independent of $n$ such that
    \begin{align}
        \int_{\End_l^{(t_n-1),t_n}} (u-c_{ln})^2\,d\Vj \leq C\int_{\End_l^{(t_n-1),t_n}} |\gradj u|^2 \, d\Vj,
    \end{align}
    where $c_{ln}$ are the averages of $u$ over the annuli. It follows from \eqref{eqn:uniformholder} that $\{c_{ln}\}_n$ is uniformly bounded and hence admits a convergent subsequence $c_{ln_i}\to c_l$. Then \eqref{eqn:uniformsobolev} and Sobolev embedding imply
    \begin{align}
        \sup_{\End_l^{(t_{n_i}-1),t_{n_i}}} |u-c_l| = o(1) \quad \text{as} \quad i\to\infty.
    \end{align}
    The maximum principle may then be applied to obtain the same conclusion on the intermediate regions $\End_l\vphantom{x}^{t_{n_i},(t_{n_{i+1}}-1)}$. This yields the full limit $|u-c_l|=o(1)$ as $|t|\to\infty$, concluding the proof of (3).

    To show (2), let $s_n$ be a sequence tending to infinity sufficiently slowly with $s_n<t_n$ so that (1) yields $|\gradj^a(u-u_n)|=o(1)$ when restricted to $\End_l^{s_n}$ for $a=0,1,2$. Then $u_n$ is a sequence of functions defined on $\smash{\End_l\vphantom{x}^{s_n,t_n} \cup \til B\vphantom{x}^{t_n}_l}$ solving the boundary value problem
    \begin{align}
        \begin{cases}
            \Delta_{\gt_{t_n}} u_n - \langle \Et, \gradt u_n\rangle_{\gt_{t_n}} = 0 & \text{in} \quad \End_l^{s_n,t_n} \cup \til B^{t_n}_l, \\
            |\gradj^a (u_n-c_l)|=o(1) & \text{on} \quad \del(\End_l^{s_n,t_n} \cup \til B^{t_n}_l) = \End_l^{s_n}, \quad a=0,1,2,
        \end{cases}
    \end{align}
    where the boundary condition follows from the triangle inequality. Then (2) is a consequence of standard interior gradient estimates.
\end{proof}

\subsubsection{Mass-charge inequality}

We now estimate the mass-charge gap on the charged Jang manifold $(\Mj,\gj,\Ej)$ in terms of charged harmonic functions. Note that some non-negative remainder terms are retained, which are the object of study in \Cref{sec:mass-charge-angular-momentum}.

\begin{theorem}\label{thm:masscharge}
    Let $(M,g,k,\E)$ be a charged, asymptotically flat or asymptotically hyperboloidal initial data set that satisfies the charged dominant energy condition. If $(\Mj,\gj,\Ej)$ is the charged Jang graph given by \Cref{prop:jang} \& \ref{prop:jangelectricfield}, then there exist charged harmonic functions $u^1,u^2,u^3$ on $\Mj$ forming asymptotically flat coordinates on $\End_0$ such that the following holds for any $i=1,2,3$ and $\frac{1}{2}\leq\varepsilon\leq\frac{3}{2}$.
    \begin{align}\label{thm:masschargeinequality}
        \jang{C}\cdot\jang E - \lp 1- \tfrac{\varepsilon}{2} \rp |\jang\Q|
        \geq \frac{1}{16\pi} \int_{\Mj} \lp 1 - \tfrac{2\varepsilon}{3}\rp\frac{|\gradc^2 u^i|^2}{|\gradj u^i|}+ \varepsilon |\gradj u^i||\Ej|^2 + |\gradj u^i| \lp |h-k|^2 + \lp 2 - \tfrac{1}{\varepsilon} \rp |X|^2 \rp \, d\Vj,
    \end{align}
    Here, all norms are taken with respect to the Jang metric, $\jang E,\jang\Q$ are respectively the ADM energy and total charge of $\End_0$, and $\jang{C}=1$ or $\frac{1}{2}$ if $M$ is asymptotically flat or asymptotically hyperboloidal.
\end{theorem}
\begin{proof}
    First consider the case when the initial data is asymptotically flat. Let $u^i$ be the charged harmonic functions on $(\Mj,\gj,\Ej)$ given by \Cref{lem:chargedharmonic} such that $u^i\to x^i$, where $x^1,x^2,x^3$ are asymptotically flat coordinates on $\End_0$. Then $U=(u^1,u^2,u^3)$ forms a coordinate system on $\End_0$ by the parallel transport argument of \cite[\S8.3]{bray2021spacetimeharmonicfunctionsapplications}. Fix $u=u^k$ and let $C_\rho\subset\End_0$ be the cylinder of radius $\rho>0$.
    \begin{align}
        C_\rho = \{U\,:\,|u^1|\leq \rho, \,(u^2)^2+(u^3)^2=\rho^2\} \cup \{ U\,:\,|u^1|=\rho, \, (u^2)^2+(u^3)^2\leq \rho^2\}
    \end{align}
    By \Cref{lem:chargedharmonic}, there exists a sequence $u_{t_n}$ of charged harmonic functions defined on the capped domains $\smash{(\til M_{t_n},\gt_{t_n})}$ satisfying \eqref{eqn:cappedchargedharmonic} with $\smash{a_i=\delta_i^k}$. For brevity, we shall write $u_n=u_{t_n}$, $\smash{(\til M_n, \gt_n)=(\til M_{t_n},\gt_{t_n})}$, and $\End_l^n = \End_l^{\pm t_n}$. Let $\Mj_{n\rho}$ be the region bounded between the levels $\{\End_l^n\}_{l=1}^{l_0}$ and the coordinate cylinder $C_\rho$ and take $\smash{\til M_{n\rho} = M_{n\rho} \bigcup_{l=1}^{l_0} \til B^{n}_l}$ to be the corresponding capped domain. We may apply the charged integral identity \cite[(8.7)]{bray2021spacetimeharmonicfunctionsapplications} on the capped domain to obtain
    \begin{equation}
        \begin{aligned}
            \int_{C_\rho} \del_\nu |u_n|_{\gt} - \til\Delta u_n \frac{\gradt_\nu u_n}{|\gradt u_n|_{\gt}} + 2|\gradt u_n|_{\gt}\langle \Et,\nu \rangle_{\gt} \, d\til A
            &\geq \frac{1}{2} \int_{\til M_{n\rho}} \Bigg[ \frac{|\gradc^2 u_n|_{\gt}^2}{|\gradt u_n|_{\gt}} + |\gradt u_n|_{\gt} \lp R_{\gt}-2|\Et|_{\gt}^2 \rp \\
            &\qquad\qquad\quad - 2 K |\gradt u_n|_{\gt} + 2 |\gradt u_n|_{\gt} \Div_{\gt} \Et \Bigg] \, d\til V,
        \end{aligned}
    \end{equation}
    where $\nu$ is the unit outer normal to $C_\rho$ and $K$ is the Gauss curvature of the level sets of $u_n$ in $\til M_{n\rho}$. According to \Cref{prop:jang} (1), the Jang surface satisfies $H_2(\til M_{n\infty};\Z)=0$, so we may apply \cite[(6.28)]{bray2019harmonicfunctionsmass3dimensional} to find
    \begin{align}
        \int_{C_\rho} \sum_{i=1}^3 (\gj_{ij,i}-\gj_{ii,j})\nu^j + 4|\gradj u_n|\langle \Ej,\nu\rangle \,d\Aj
        \geq \int_{\til M_{n\rho}} \frac{|\gradc^2 u_n|_{\gt}^2}{|\gradt u_n|_{\gt}} + |\gradt u_n|_{\gt} \lp R_{\gt}-2|\Et|_{\gt}^2 + 2 \Div_{\gt} \Et \rp \, d\til V + o(1),
    \end{align}
    where the remainder vanishes in the limit $\rho\to\infty$. Separating out the integral over the caps, we may use that the Jang electric field is divergence-free on the Jang surface to obtain
    \begin{equation}\label{thm:masscharge:step1}
        \begin{aligned}
            \int_{C_\rho} \sum_{i=1}^3 (\gj_{ij,i}-\gj_{ii,j})\nu^j + 4|\gradj u_n|\langle \Ej,\nu\rangle \,d\Aj
            &\geq \int_{\Mj_{n\rho}} \frac{|\gradc^2 u_n|^2}{|\gradj u_n|} + |\gradj u_n| \lp R_{\gj}-2|\E|_g^2 \rp \, d\Vj \\
            &\quad + \sum_{l=1}^{l_0} \underbrace{\int_{\til B_l^n} |\gradt u_n|_{\gt} \lp R_{\gt}-2|\Et|_{\gt}^2 + 2 \Div_{\gt} \Et \rp \, d\til V}_{\equiv\I_{ln}^1}, 
        \end{aligned}
    \end{equation}
    where the original electric field appears in the first integral after applying \Cref{prop:jangelectricfield}. It follows from \Cref{lem:chargedharmonic} (2), the uniformly bounded geometry of the caps, and the uniform bounds on $|\Et|_{C^0(\til M_n)},|\Et|_{C^1(\til M_n)}$ that $\lim_{n\to\infty} \I^1_{ln} = 0$. Now apply the Jang scalar curvature formula \eqref{eqn:jangscalarcurvature} on the right-hand side and integrate the divergence terms by parts to get
    \begin{equation}\label{thm:masscharge:step2}
        \begin{aligned}
            \int_{\Mj_{n\rho}} \frac{|\gradc^2 u_n|^2}{|\gradj u_n|} + |\gradj u_n|\lp |h-k|^2 + 2|X|^2\rp - 2 \langle X,\gradj|\gradj u_n|\rangle \, d\Vj \\
            - \sum_{l=1}^{l_0} \underbrace{\int_{\End_l^n} 2\langle X,\upsilon \rangle |\gradj u_n| \, d\Aj}_{\equiv \I^2_{ln}}
            - \underbrace{\int_{C_\rho} 2\langle X,\nu \rangle |\gradj u_n| \, d\Aj}_{\equiv \I^3_{n\rho}} 
        \end{aligned}
    \end{equation}
    where $\upsilon$ is the unit outer normal to $\End_l^n$. Then \Cref{lem:chargedharmonic} (2) and \Cref{prop:jang} (7) imply $\lim_{n\to\infty} \I^2_{ln} = 0$. In addition, the proof of \Cref{lem:chargedharmonic} shows $|\gradj u_n|$ is bounded independent of $n$ and $\rho$, while \Cref{prop:jang} (6) gives $|X|_{\gj}=O(r^{-1-2q})$. Since $|\rho-r|=o(1)$ as $\rho\to\infty$, it follows that $\I^3_{n\rho} \to 0$ uniformly in $n$ as $\rho\to\infty$. The remaining inner product may be estimated by Young's inequality and \Cref{lem:kato}.
    \begin{equation}
        \begin{aligned}\label{thm:masscharge:step3}
            |\langle X,\gradj|\gradj u_n|\rangle|
            &\leq \frac{\varepsilon}{2}\frac{|\gradj|\gradj u_n||^2}{|\gradj u_n|} + \frac{1}{2\varepsilon}|\gradj u_n||X|^2 \\
            &\leq \frac{\varepsilon}{3} \frac{|\gradc^2 u_n|^2}{|\gradj u_n|} - \varepsilon\Div_{\gj}(|\gradj u_n|\Ej) - \frac{\varepsilon}{2}|\gradj u_n||\Ej|^2 + \frac{1}{2\varepsilon}|\gradj u_n||X|^2 
        \end{aligned}
    \end{equation}
    Hence \eqref{thm:masscharge:step2} may be bounded below by
    \begin{align}
        \int_{\Mj_{n\rho}} \lp 1 - \tfrac{2\varepsilon}{3}\rp\frac{|\gradc^2 u_n|^2}{|\gradj u_n|} + \varepsilon |\gradj u_n||\Ej|^2 + |\gradj u_n| |h-k|^2 + \lp 2 - \tfrac{1}{\varepsilon} \rp |\gradj u_n| |X|^2 + 2\varepsilon \Div_{\gj}(|\gradj u_n|\Ej) \,d\Vj.
    \end{align}
    By the divergence theorem, the last term is given by
    \begin{align}\label{thm:masscharge:step4}
        \int_{\Mj_{n\rho}} 2\varepsilon \Div_{\gj}(|\gradj u_n|\Ej) \, d\Vj
        = \int_{C_\rho} 2\varepsilon|\gradj u_n|\langle\Ej,\nu\rangle\,d\Aj
        - \sum_{l=1}^{l_0} \underbrace{\int_{\End_l^n} 2\varepsilon|\gradj u_n|\langle\Ej,\nu\rangle\,d\Aj}_{\equiv\I_{ln}^4},
    \end{align}
    where \Cref{prop:jangelectricfield} and \Cref{lem:chargedharmonic} (2) imply $\lim_{n\to\infty}\I_{ln}^4=0$. We may combine the estimates \eqref{thm:masscharge:step1}--\eqref{thm:masscharge:step4} and apply Fatou's lemma in the resulting inequality to take the limit $\rho\to\infty$ and liminf as $n\to\infty$. This concludes the proof in the asymptotically flat case.

    Now suppose the initial data is asymptotically hyperboloidal. The previous argument works, with slight adjustments to accommodate weaker decay in the Jang triple. Let $x^i$ be the coordinates used in \cite[Corollary 6.11]{sakovich2021jangequationpositivemass} and let $u^1,u^2,u^3$ be the corresponding charged harmonic functions solving \eqref{eqn:cappedchargedharmonic}. Since $|u^i-x^i|=O_2(r^0)$, either $x^i$ or $u^i$ may be used to compute the ADM energy. As before, fix $u=u^k$ and let $u_n$ be the approximating sequence. Then we may apply \eqref{thm:masscharge:step1}--\eqref{thm:masscharge:step4} to get
    \begin{align}
        \int_{C_\rho} \sum_{i=1}^3 (\gj_{ij,i}-\gj_{ii,j})\nu^j + (4-2\varepsilon)|\gradj u_n|\langle \Ej,\nu\rangle \,d\Aj
        &\geq \int_{\Mj_{n\rho}} \Bigg[\!\lp 1 - \tfrac{2\varepsilon}{3}\rp\frac{|\gradc^2 u_n|^2}{|\gradj u_n|}+ \varepsilon |\gradj u_n||\Ej|^2 + |\gradj u_n||h-k|^2 \nonumber \\
        &\qquad + \lp 2 - \tfrac{1}{\varepsilon} \rp |\gradj u_n| |X|^2 \Bigg] \,d\Vj
        + \sum_{l=1}^{l_0} \lp \I_{ln}^1 - \I_{ln}^2 - \I_{ln}^4 \rp - \I_{n\rho}^3. \label{thm:masscharge:eqn4}
    \end{align}
    For the same reasons as before, $\lim_{n\to\infty} \I^1_{ln}-\I^2_{ln} = 0$. On the other hand, since \Cref{prop:jang} (6) gives $|X|_{\gj}=O(r^{-2})$, the boundary integral $\I^3_{n\rho}$ has a nonzero limit, given by
    \begin{align}
        \I_{n\rho}^3 = \int_{C_\rho} 2\langle X,\nu \rangle |\gradj u_n| \, d\Aj
        = \int_{C_\rho} 2\langle X,\nu \rangle \, d\Aj + o(1).
    \end{align}
    Recall the following coordinate expression for scalar curvature
    \begin{align}
        R_{\gj} = \del_j \lp\gj^{ij} \gj^{kl} ( \gj_{li,k} - \gj_{lk,i})\rp + Q(\del\gj),
    \end{align}
    where $Q(\del\gj)$ is quadratic in the first derivatives of the metric. In charged harmonic coordinates, we have $Q(\del\gj)=O(r^{-4})$. Thus on the region $D_\rho$ bounded between the cylinder $C_\rho$ and the sphere $S_{2\rho}$, we may integrate by parts to get
    \begin{equation}
        \begin{aligned}
            \int_{\del D_\rho} \sum_{i=1}^3 (\gj_{ij,i}-\gj_{ii,j})\nu^j \, d\Aj
            &= \int_{D_\rho} R_{\gj} \, d\Vj + o(1) \\
            &= \int_{D_\rho} R_{\gj} + 2\Div_{\gj} X \, d\Vj - \int_{S_{2\rho}} 2\langle X,\nu\rangle \, d\Aj + \int_{C_\rho} 2\langle X,\nu\rangle \, d\Aj + o(1) \\
            &= - \int_{S_{2\rho}} 2\langle X,\nu\rangle \, d\Aj + \I_{n\rho}^3 + o(1),
        \end{aligned}
    \end{equation}
    where the last equality uses the integrability of $R_{\gj}+2\Div_{\gj}X$ given by \Cref{prop:jang} (3). Combined with \eqref{thm:masscharge:eqn4}, we have
    \begin{align}
        \int_{S_{2\rho}} \sum_{i=1}^3 (\gj_{ij,i}-\gj_{ii,j})\nu^j + 2\langle X,\nu\rangle \, d\Aj + \int_{C_\rho} (4-2\varepsilon)|\gradj u|\langle \Ej,\nu\rangle  \, d\Aj
        &\geq \int_{\Mj_{n\rho}} \lp 1 - \tfrac{2\varepsilon}{3}\rp\frac{|\gradc^2 u_n|^2}{|\gradj u_n|}+ \varepsilon |\gradj u_n||\Ej|^2 \nonumber \\
        &\qquad\quad + |\gradj u_n| \Big\{ |h-k|^2 + \lp 2 - \tfrac{1}{\varepsilon} \rp |X|^2 \Big\} d\Vj.
    \end{align}
    According to \Cref{prop:jang} (6), the flux of $2X$ converges to $-8\pi E_\mathrm{ADM}(\End_0)$ in the limit $\rho\to\infty$. The desired inequality by using Fatou's lemma to take the limit $r\to\infty$ and the liminf as $n\to\infty$.
\end{proof}

\section{Mass-charge-angular momentum inequality}\label{sec:mass-charge-angular-momentum}

Let $(M,g,k)$ be an axially symmetric initial data set with axial Killing field $\eta$. The \emph{Komar angular momentum} of a surface $\Sigma\subset M$ is given by 
\begin{align}
    \J(\Sigma) = \frac{1}{8\pi} \int_\Sigma k(\eta,\nu) \, dA_g,
\end{align}
where $\nu$ is the outer unit normal to $\Sigma$. The conservation property \eqref{cond:conservationofangularmomentum} and the asymptotic conditions in \eqref{cond:AF_asymptotics}, \eqref{cond:AH_asymptotics} imply that $\J(\Sigma)$ only depends on the homology class of $\Sigma$. In particular, the Komar momentum of any surface homologous to a sufficiently large coordinate sphere in the asymptotic end recovers the total angular momentum. 

If $\Sigma$ is axisymmetric, then the orbit of $\eta$ passing through $p\in\Sigma$ has circumference $2\pi|\eta(p)|$. Thus the quantity $\int_\Sigma|\eta|^2\,dA$ reflects the distribution of area of an axisymmetric surface relative to the axis of symmetry. For any foliation $\{\Sigma_t\}_{t\geq 0}$ of an end by axisymmetric surfaces, define the \emph{axial radius} of $\Sigma_0$ to be
\begin{align}\label{def:radius}
    \mathcal{R}_\eta(\Sigma_0) = \lp \int_0^\infty \frac{8\pi}{\int_{\Sigma_t} |\eta|_g^2 \, dA_t} \, dt \rp^{-1/3}.
\end{align}
The normalizing factor is chosen here so that if $\{S_r\}_r\subset\R^3$ consists of round spheres parametrized by their radius, then $\mathcal{R}_\eta(S_r)=r$. The definition \eqref{def:radius} is similar to the length $\mathcal{R}$ introduced in \cite{anglada2018,anglada2020}, which recovers for an inverse mean curvature flow $\{\Sigma_t\}_t$ the areal radius $\mathcal{R}(\Sigma_t) = \sqrt{|\Sigma_t|/4\pi}$ whenever the flow is smooth and spherical. 

Our first step is to construct coordinates adapted to the axisymmetry of the manifold while preserving the asymptotic conditions of the end. In particular, we shall require the coordinate spheres to be axisymmetric.

\begin{lemma}\label{lem:axisymmetriccoordinates}
    Let $(M,g,k,\E)$ be a charged, axisymmetric initial data set that is either asymptotically flat or asymptotically hyperboloidal with axial Killing field $\eta$. After possibly restricting $\Mend$ to a smaller asymptotic end, there exist $R_0>0$ and coordinates $(\ax r,\ax\theta,\ax\phi)$ on $\Mend\cong(R_0,\infty)\times S^2$ such that:
    \begin{enumerate}
        \item The Killing field is given by $\eta=\del_{\ax\phi}$.
        \item If the initial data is asymptotically flat, then it satisfies \eqref{cond:AF_asymptotics} in the coordinates
        \begin{align*}
            \ax x^1=\ax r\sin\ax\theta\cos\ax\phi, \qquad
            \ax x^2=\ax r\sin\ax\theta\sin\ax\phi, \qquad 
            \ax x^3=\ax r\cos\ax\theta.
        \end{align*}
        \item If the initial data is asymptotically hyperboloidal, then it satisfies \eqref{cond:AH_asymptotics} in the coordinates $\smash{(\ax r,\ax\theta,\ax\phi)}$, with the error term in the first equation replaced by $O_5(\ax{r}^{-2})$. Moreover, the associated total energy $\ax{E}$ satisfies $\ax{E}\leq E$.
    \end{enumerate}
\end{lemma}
\begin{proof}
    We first consider the asymptotically flat case. After restricting to an asymptotic end $\Mend$, we may apply the result of \cite{chrusciel2008massangularmomentuminequalities}, which gives Brill coordinates $(\rho,\phi,z)$ such that $\eta=\del_\phi$ and
    \begin{align}\label{eqn:brill}
        g = e^{2(\alpha-U)}(d\rho^2+dz^2) + \rho^2e^{-2U}(d\phi+A_\rho d\rho + A_zdz)^2,
    \end{align}
    where $\alpha,U,A_\rho,A_z$ depend only on $(\rho,z)$ and satisfy
    \begin{align}\label{eqn:brillasymptotics}
        \alpha, U = O_1(r^{-q}), \qquad A_\rho,A_z = O_2(r^{-q-1}), \qquad r=\sqrt{\rho^2+z^2.}
    \end{align}
    Define $\ax r = r$, $\ax\theta = \arccos{\frac{z}{r}}$, $\ax\phi=\phi$. Then the Brill asymptotics imply that the associated Cartesian coordinates $(\ax{x}^1,\ax{x}^2,\ax{x}^3)$ form an asymptotically flat coordinate system on the end.

    Now suppose $(M,g)$ is asymptotically hyperboloidal. Let $\{\Psi_t\}_{t\in S^1}$ denote the one-parameter family of isometries generated by $\eta$, with $\Psi_0 = \text{Id}$. Introduce the defining function $\varrho =2(r+\sqrt{1+r^2})^{-1}$. The asymptotic expansion for $g$ in \eqref{cond:AH_asymptotics} may then be written as
    \begin{align}\label{lem:axisymmetriccoordinates:compactification}
        g = \frac{1}{\varrho^2} \ls d\varrho^2 + \lp 1 - \frac{\varrho^2}{4} \rp^2 \sigma + \varrho^3 \mathbf{m} + O_6(\varrho^4) \rs.
    \end{align}
    Thus $(M,g)$ admits a conformal compactification obtained by adjoining the sphere $\{\varrho=0\}\equiv \del_\infty M$, whose induced conformal representative is the standard round metric $\sigma$. Each $\Psi_t$ extends to the conformal boundary and its restriction $\Psi_t^\del\equiv\Psi_t\vert_{\del_\infty M}$ is conformal with respect to $\sigma$ by \cite[Theorem 6.1]{chrusciel2003massasymptoticallyhyperbolicriemannian}; see Remark 6.2 there regarding compactifications of finite regularity.

    Fix an orientation on $\del_\infty M$. Since $\det{d\Psi_t^\del}$ depends continuously on $t$ and $\det{d\Psi^\del_0}=1$, each $\Psi^\del_t$ is orientation preserving. Then since $\{\Psi_t^\del\}_t$ is the image of a continuous homomorphism on $S^1$, it is a compact subgroup of $\mathrm{Conf}^+(S^2) \cong \mathrm{Isom}^+(\mathbb{H}^3)$, where $\mathrm{Conf}^+(S^2)$ consists of the orientation-preserving conformal transformations of $S^2$ and $\mathrm{Isom}^+(\mathbb{H}^3)$ consists of the orientation-preserving isometries of $\HS$. By Cartan's fixed point theorem, this subgroup fixes some point $q\in\HS$. Choose a hyperbolic isometry $A$ satisfying $A(q)=0$. Then $\{A\Psi_t^\del A^{-1}\}_t\subset\mathrm{SO}(3)$. Since every one-dimensional compact subgroup of $\mathrm{SO}(3)$ consists of rotations about a fixed axis, after composing $A$ with an ordinary rotation if necessary, we may assume that $\{A\Psi_t^\del A^{-1}\}_t$ consists of rotations about the $z$-axis. 

    Let $\Phi$ denote the original asymptotically hyperboloidal chart. In the straightened chart $A\circ\Phi$, write the energy-momentum vector as $(E',P')$. The rotational symmetry about the $z$-axis implies that $P'=(0,0,P'_3)$. We claim the original energy-momentum vector is timelike, $E>|P|$. Observe that the charged dominant energy condition implies the uncharged dominant energy condition, since $\mu\equiv\mu_E+\smash{\frac{|\E|^2}{8\pi}}\geq|J|$. Hence if $E=|P|$, the rigidity statement for the hyperboloidal positive mass theorem \cite{hirsch2026hyperboloidalspacetimepositivemass} implies that $(M,g,k)$ isometrically embeds as a spacelike hypersurface of Minkowski space. Such a hypersurface is diffeomorphic to $\R^3$ by \cite[Proposition 9.3]{alias2016maximum}. If $\Sigma$ is a connected component of an apparent horizon in $M$, then $\Sigma$ is in particular a closed, two-sided, embedded surface in $\R^3$, and hence separates $\R^3$ by the Jordan-Brouwer separation theorem. This contradicts \cite[Corollary 2]{carrasco2008marginallyoutertrappedsurfaces}, which states that an initial data set for Minkowski space cannot contain a bounding MOTS/MITS. Thus the total mass of $\Mend$ is strictly positive. Moreover, since mass is invariant under Lorentz transformations, $m=m'=\sqrt{(E')^2-(P'_3)^2}>0$.

    Take the asymptotically hyperboloidal chart on $\Mend$ and apply a boost $B_\beta$ in the $z$-direction with speed $|\beta|<1$. Since boosts in the $z$-direction commute with rotations about the $z$-axis, the rotational symmetry of the energy-momentum vector is preserved. The energy-momentum components in the boosted chart $B_\beta\circ A\circ\Phi$ are given by
    \begin{align}
        \ax{E}=\frac{E'-\beta P'_3}{\sqrt{1-\beta^2}}, \qquad \ax{P}_3=\frac{P'_3-\beta E'}{\sqrt{1-\beta^2}}.
    \end{align}
    By the previous paragraph. Hence we may choose $\beta=P'_3/E'$, which gives $\ax{E}=m$, $\ax{P}_3=0$. In particular, the energy in this chart is no larger than the original energy.

    Set $\Lambda=B_\beta\circ A$ and let $\ax\sigma=\Lambda^*\sigma$ denote the corresponding round representative of the conformal metric at infinity. Since $\Lambda\Psi_t^\del \Lambda^{-1}$ is a rotation, $\Psi_t^\del$ is an isometry of $(\del_\infty M,\ax\sigma)$.
    \begin{align}
        (\Psi_t^\del)^*\ax\sigma
        = (\Lambda\Psi_t^\del)^*\sigma
        = \Lambda^*(\Lambda\Psi_t^\del \Lambda^{-1})^*\sigma
        = \ax\sigma
    \end{align}
    By \cite[Lemma 2.2]{wang2001massasymptoticallyhyperbolicmanifolds}, there is a unique defining function $\ax\varrho$ satisfying
    \begin{align}\label{lem:axisymmetriccoordinates:definingfunction}
        |d\ax\varrho|_{\ax g}=1, \qquad
        \ax g\vert_{T\del_\infty M}=\ax\sigma,
    \end{align}
    where $\ax{g}=\ax{\varrho}^2 g$. For each $t$, set $\varrho_t=\ax\varrho\circ\Psi_t$. Since $\Psi_t$ is an isometry of $g$,
    \begin{align}
        |d\varrho_t|^2_{\varrho_t^2 g}
        = |\Psi_t^*d\ax\varrho|^2_{\Psi_t^*\ax g}
        = |d\ax\varrho|^2_{\ax g}\circ\Psi_t
        = 1.
    \end{align}
    Moreover, at conformal infinity we have
    \begin{align}
        (\varrho_t^2 g)\vert_{T\del_\infty M}
        = (\Psi_t^*\ax g) \vert_{T\del_\infty M}
        = (\Psi_t^\del)^*\ax\sigma
        = \ax\sigma.
    \end{align}
    Thus $\varrho_t$ satisfies the same defining properties as $\ax\varrho$, and by uniqueness,
    \begin{align}\label{lem:axisymmetriccoordinates:uniqueness}
        \varrho_t = \ax\varrho\circ\Psi_t = \ax\varrho.
    \end{align}
    Differentiating at $t=0$ yields $\eta(\ax\varrho)=0$. Consequently, if $\ax r=\frac{1}{\ax\varrho}-\frac{\ax\varrho}{4}$, then we also have $\eta(\ax r)=0$. 

    It remains to construct the symmetry-adapted angular coordinates. By \eqref{lem:axisymmetriccoordinates:definingfunction}, the gradient vector field $N=\grad_{\ax{g}}\ax\varrho$ is the inward unit normal to the level sets of $\ax\varrho$ and its flow $\{F_s\}_{0\leq s<\varepsilon}$ defines a diffeomorphism between $[0,\varepsilon)\times S^2$ and a collared neighborhood $\mathcal{W}$ of conformal infinity. Let $\smash{(\ax\theta,\ax\phi)=(\theta\circ\Lambda,\phi\circ\Lambda)}$ be angular coordinates at the boundary and extend them into $\mathcal{W}$ by requiring them to remain constant along the flow lines of $N$. It follows from \eqref{lem:axisymmetriccoordinates:uniqueness} that $\Psi_t$ is an isometry of $\ax{g}$ and hence satisfies $d\Psi_t(N)=N$. By the uniqueness of integral curves, the normal flow commutes with the axial action, $\Psi_t \circ F_s = F_s\circ\Psi^\del_t$. Thus throughout $\mathcal{W}$, we have
    \begin{align}
        \Psi_t(\ax\varrho,\ax\theta,\ax\phi)=(\ax\varrho,\ax\theta,\ax\phi+t).
    \end{align}
    Differentiating $\Psi_t$ at $t=0$ gives $\smash{\eta=\del_{\ax\phi}}$. We next verify that the symmetry-adapted coordinates $\smash{(\ax{r},\ax\theta,\ax\phi)}$ form asymptotically hyperboloidal coordinates on $\mathcal{W}\subset\Mend$. The computation of \cite[Section 3]{wang2001massasymptoticallyhyperbolicmanifolds} shows that the asymptotic class for $g$ in \eqref{cond:AH_asymptotics} is preserved under a change of the conformal representative at infinity. Hence there is a symmetric 2-tensor $\ax{\mathbf{m}}$ on $S^2$ such that $\smash{g=\ax{b}+\frac{\ax{\mathbf{m}}}{\ax r}+O_5(\ax r^{-2})}$ as $\ax{r}\to\infty$, where $\smash{\ax{b}=\frac{d\ax r^2}{1+\ax r^2}+\ax r^2\sigma}$. 
    
    Since $\varrho$ and $\ax\varrho$ are defining functions whose boundary metrics are conformally related, there exists a function $\omega$ such that $\ax\varrho=e^\omega\varrho$. By \eqref{lem:axisymmetriccoordinates:compactification}, $|d\varrho|_{\varrho^2g}^2=1+O(\varrho^4)$. Expanding the geodesic defining function equation $|d\ax\varrho|_{\ax{g}}^2=1$ therefore gives 
    \begin{align}
        1 
        = |d\ax\varrho|_{\ax g}^2 
        = |\varrho d\omega+d\varrho|_{\varrho^2g}^2
        = \varrho^2|d\omega|_{\varrho^2g}^2 + 2\varrho\del_\varrho \omega + 1 + O(\varrho^4).
    \end{align}
    Since $d\omega$ remains bounded, we may divide by $\varrho$ and evaluate at $\varrho=0$ to get $\del_\varrho\omega|_{\varrho=0}=0$. Thus
    \begin{align}\label{lem:axisymmetriccoordinates:varrhoexpansion}
        \ax\varrho=e^{\omega_0}\varrho+O(\varrho^3) \quad\text{as}\quad \varrho\to 0,
    \end{align}
    where $\omega_0=\omega\vert_{\varrho=0}$ only depends on $(\theta,\phi)$. From the definitions of $r$ and $\ax{r}$,
    \begin{align}
        \frac{1}{r}=\frac{1}{\ax{r}e^{\omega_0}}+O(\ax\varrho^3).
    \end{align}
    Hence there exists $C\geq 1$ such that $\frac{1}{C} r\leq\ax r\leq Cr$ for all sufficiently large $r$. It follows that $\mu_E,J=O_3(\ax{r}^{-3-q_0})$, $\E=O_1(\ax{r}^{-2})$. To compare the angular coordinates, write the normal field $N$ in the original coordinates $(\varrho,\theta,\phi)$. Using \eqref{lem:axisymmetriccoordinates:compactification}, we have
    \begin{align}
        N
        =\grad_{\ax g}\ax\varrho
        =e^{-2\omega}\grad_{\varrho^2g}(e^\omega\varrho)
        =e^{-\omega}\lp\del_\varrho+\varrho\grad_{\varrho^2 g}\omega + O(\varrho^4)\rp.
    \end{align}
    Thus the component of $N$ tangent to the $S^2$ factor is $O(\varrho)$. Along an integral curve of $N$,
    \begin{align}
        \frac{\del\theta}{\del\ax\varrho}=\frac{\del\phi}{\del\ax\varrho}=O(\ax\varrho).
    \end{align}
    Integrating from the boundary relates $y=(\theta,\phi)$ and $\ax{y}=(\ax\theta,\ax\phi)$ by
    \begin{align}\label{lem:axisymmetriccoordinates:angularexpansion}
        y=\ax{y}\circ\Lambda^{-1}+O(\ax\varrho^2).
    \end{align}
    Then differentiating yields
    \begin{align}\label{lem:axisymmetriccoordinates:angulardifferentialexpansion}
        dy^i
        =(\Lambda^{-1})^*(d\ax{y}^i) + O(\ax\varrho^2)(d\ax\theta+d\ax\phi)+O(\ax\varrho)d\ax\varrho.
    \end{align}
    Lastly, the original asymptotics give $k-g=\frac{\mathbf{s}}{r}+O(r^{-2})$, where $\mathbf{s}=\mathbf{p}-\mathbf{m}$ is a symmetric 2-tensor on $S^2$. Combining \eqref{lem:axisymmetriccoordinates:varrhoexpansion} and \eqref{lem:axisymmetriccoordinates:angulardifferentialexpansion}, we obtain
    \begin{align}
        \frac{\mathbf{s}}{r}
        = \frac{(\Lambda^{-1})^*\mathbf{s}}{\ax{r}e^{\omega_0}} + O(\ax{r}^{-2}).
    \end{align}
    The required decay of the remainder term may be established by carefully tracking the regularity loss through the preceding coordinate changes. Setting $\smash{\mathbf{\ax{s}}=e^{-\omega_0}(\Lambda^{-1})^*\mathbf{s}}$ and $\mathbf{\ax{p}}=\mathbf{\ax{m}}+\mathbf{\ax{s}}$ gives $\smash{k=\ax{b}+\frac{\mathbf{\ax{p}}}{\ax{r}}+O_4(\ax{r}^{-2})}$, completing the proof. 
\end{proof}

Let $(\ax r,\ax\theta,\ax\phi)$ be the symmetry-adapted coordinates from \Cref{lem:axisymmetriccoordinates}. If $(M,g)$ is asymptotically flat, the asymptotics of the metric imply
\begin{align}
    |\eta|_g^2 = \ax r^2 \sin^2{\ax\theta}(1+O(\ax r^{-q})) \quad \text{as} \quad\ax r\to\infty,
\end{align}
along the end. The axial radius of $S_{\ax{r}}=\{\ax{r}=\text{const}\}$ is then related to the coordinate radius by
\begin{align}
    \frac{\mathcal{R}_\eta(S_{\ax r})}{\ax r} = 1 + O(\ax r^{-q}),
\end{align}
where $q>1/2$. Similarly, in the asymptotically hyperboloidal setting, the Wang asymptotics imply
\begin{align}
    |\eta|_g^2 = \ax r^2\sin^2{\ax\theta} + O(\ax r^{-1}), \qquad
    \frac{\mathcal{R}_\eta(S_{\ax r})}{\ax r} = 1 + O(\ax r^{-3})\qquad \text{as} \quad\ax r\to\infty.
\end{align}

Using \Cref{lem:axisymmetriccoordinates}, we may apply the earlier existence results for the Jang equation on $\Mend$. In particular, the asymptotic expansions in $\smash{(\ax{r},\ax\theta,\ax\phi)}$ have the regularity required for the analysis of \cite{sakovich2021jangequationpositivemass}. We obtain a Jang solution satisfying $\smash{f=O_3(\ax{r}^{1-2q})}$ in the asymptotically flat case and $\smash{f=\sqrt{1+\ax{r}^2}+\ax\alpha\log{\ax{r}}+\ax\psi+O_3(\ax{r}^{-1+\ax\varepsilon})}$ in the asymptotically hyperboloidal case, where $\ax\varepsilon\in(0,1)$, $\smash{\ax\alpha,\ax\psi\in C^3(S^2)}$. Additionally, the associated Jang graph $(\Mj,\gj)$ admits the properties listed in \Cref{table:jang}.

We now establish a preliminary estimate for the remainder term in \Cref{thm:masscharge}.

\begin{lemma}\label{lem:geometricangularmomentum}
    Let $(M,g,k)$ be an axisymmetric initial data set that is either asymptotically flat or asymptotically hyperboloidal. Suppose its axial Killing field $\eta$ satisfies $\eta^\flat\wedge d\eta^\flat =0$. Then for any $R\geq R_0$,
    \begin{align}
        \int_{\{\bar{r}\geq R\}} |\gradj\bar{r}|_{\gj}\lp|h-k|^2_{\gj} + |X|_{\gj}^2\rp  d\Vj 
        \geq \frac{8\pi \J^2}{\mathcal{R}_\eta(S_R)^3},
    \end{align}
    where $\bar{r}$ is the pullback of $\ax r$ to the Jang surface and $S_R=\{\ax{r}=R\}$ is an axisymmetric coordinate sphere.
\end{lemma}
\begin{remark}
    The condition $\eta^\flat\wedge d\eta^\flat=0$ is the coordinate-invariant version of the simplifying assumption $A=B=0$ in \cite[Appendix A]{jaracz2018bekensteinboundspenroseinequalities}.
\end{remark}
\begin{proof}
    Consider the pullback $\bar{r} = \ax r\circ\pi$ by the graph projection $\pi:\Mj\to M$. By the co-area formula,
    \begin{align}\label{lem:geometricangularmomentum:eqn1}
        \int_{\{\bar{r}\geq R\}} |\gradj\bar{r}|_{\gj}\lp|h-k|^2_{\gj} + |X|_{\gj}^2\rp \, d\Vj 
        = \int_R^\infty \lp \int_{\{\bar{r}=\ax r\}} |h-k|^2_{\gj} + |X|_{\gj}^2 \, d\Aj \rp d\ax r.
    \end{align}
    Let $\nu$ be the unit outer normal to a coordinate sphere $S_{\ax r}\subset M$. By the axisymmetry of $f$,
    \begin{align}
        h(\eta,\nu)
        = \frac{\nu(\eta(f))-df(\grad_\nu\eta)}{\sqrt{1+|\grad f|_g^2}} 
        = -\frac{\langle\grad f,\grad_\nu\eta\rangle_g}{\sqrt{1+|\grad f|^2}} 
        = -\frac{d\eta^\flat(\nu,\grad f)}{2\sqrt{1+|\grad f|_g^2}},
    \end{align}
    where the last equality follows from $\eta$ being a Killing field. For any $V,W\in\eta^\perp$, the 1-form $\eta^\flat$ satisfies
    \begin{align}
        |\eta|^2 d\eta^\flat(V,W) = (\eta^\flat\wedge d\eta^\flat)(\eta,V,W).
    \end{align}
    Since $\nu\perp\eta$ by the axisymmetry of the sphere, the twist-free assumption $\eta^\flat\wedge d\eta^\flat=0$ implies
    \begin{align}\label{lem:geometricangularmomentum:eqn2}
        h(\eta,\nu)=0.
    \end{align}
    Next we claim that
    \begin{align}\label{lem:geometricangularmomentum:eqn3}
        (h-k)(\eta,\nu)^2 \leq |\eta|_g^2 \lp |h-k|_{\gj}^2 + |X|_{\gj}^2 \rp.
    \end{align}
    To show this, let $e_i$ be a local orthonormal basis with $\smash{e_1=\frac{\grad f}{|\grad f|_g}}$, $\smash{e_2 = \frac{\eta}{|\eta|_g}}$ and denote $T=h-k$. In this basis, the inverse Jang metric is given by $\smash{\gj^{-1}=\text{diag}(W^{-2},1,1)}$ where $W=\sqrt{\smash[b]{1+|\grad f|_g^2}}$. We compute
    \begin{equation}
        \begin{aligned}
            |T|_{\gj}^2
            &= \frac{1}{W^2}T_{11}^2 + T_{22}^2 + T_{33}^2 + \frac{2}{W^2}(T_{12}^2+T_{13}^2)+2T_{23}^2 \\
            |X|_{\gj}^2,
            &= \frac{|\grad f|_g^2}{W^2} |T(e_1,\cdot)|_{\gj}^2
            = \frac{|\grad f|_g^2}{W^2} \lp \frac{1}{W^2}T_{11}^2 + T_{12}^2 + T_{13}^2 \rp,
        \end{aligned}
    \end{equation}
    which implies the lower bound
    \begin{align}
        |T|_{\gj}^2 + |X|_{\gj}^2
        \geq \frac{2+|\grad f|_g^2}{1+|\grad f|_g^2} T_{12}^2 + 2 T_{23}^2.
    \end{align}
    Observe that $\nu^2=0$ since $\nu\perp\eta$. For any $a_1,a_2>0$, the weighted Cauchy-Schwarz inequality implies
    \begin{align}
        |T(e_2,\nu)|
        = \lv\nu^1 T_{12} + \nu^3 T_{23}\rv
        \leq \lp \frac{(\nu^1)^2}{a_1} + \frac{(\nu^3)^2}{a_2} \rp^\frac{1}{2} (a_1 T_{12}^2 + a_2 T_{23}^2)^\frac{1}{2}.
    \end{align}
    Take $a_1=\frac{2+|\grad f|_g^2}{1+|\grad f|_g^2}$, $a_2=2$, which satisfy $a_1\leq a_2$. Then
    \begin{align}
        T(\eta,\nu)^2
        \leq \frac{|\eta|^2}{a_1} (a_1 T_{12}^2 + 2T_{23}^2)
        \leq |\eta|_g^2 \lp |h-k|_{\gj}^2 + |X|_{\gj}^2 \rp,
    \end{align}
    establishing the desired inequality. By Cauchy-Schwarz,
    \begin{align}
        \ls 8\pi\J(S_{\ax r}) \rs^2
        &= \lp \int_{S_{\ax r}} (k-h)(\eta,\nu) \, dA \rp^2 && \text{(Eqn.~\ref{lem:geometricangularmomentum:eqn2})}\nonumber \\
        &\leq \lp \int_{S_{\ax r}} (|h-k|_{\gj}^2+|X|^2_{\gj}) \sqrt{1+|\grad^{S_{\ax r}}f|^2} \, dA \rp \lp \int_{S_{\ax r}} \frac{|\eta|_g^2 }{\sqrt{1+|\grad^{S_{\ax r}}f|^2}} \, dA \rp && \text{(Eqn.~\ref{lem:geometricangularmomentum:eqn3})} \nonumber \\
        &\leq \lp \int_{\{\bar{r}=\ax r\}} |h-k|_{\gj}^2+|X|^2_{\gj} \, d\Aj \rp \lp \int_{S_{\ax r}} |\eta|_g^2 \, dA \rp,
    \end{align}
    where $\grad^{S_{\ax r}} f$ denotes the tangential projection to $S_{\ax r}$. Rearranging and integrating over $\ax r$ gives the result.
\end{proof}

\begin{proposition}\label{prop:masschargeangularmomentum}
    Let $(M,g,k,\E)$ be a charged, axisymmetric initial data set that is either asymptotically flat or asymptotically hyperboloidal. Suppose its axial Killing field $\eta$ satisfies $\eta^\flat\wedge d\eta^\flat=0$. Then there exist charged harmonic functions $u^1,u^2,u^3$ on the charged Jang graph $(\Mj,\gj,\Ej)$ forming asymptotically flat coordinates on the end $\End_0$ such that the following holds.
    \begin{align}\label{prop:massangularmomentum:inequality}
        \jang C\cdot\jang E-\frac{7}{10}|\jang\Q| - \frac{\J^2}{36\mathcal{R}_\eta(S_{R_*})^3} \geq \frac{1}{80\pi} \sum_{i=1}^3 \int_{\Mj} \frac{|\gradc^2 u^i|^2}{|\gradj u^i|}+|\gradj u^i||\Ej|^2 \, dV_{\gj},
    \end{align}
    Here, $\jang E,\jang\Q$ are respectively the ADM energy and total charge of $\End_0$, $\jang{C}=1$ or $\frac{1}{2}$ if $M$ is asymptotically flat or asymptotically hyperboloidal, and $S_{R_*}=\{\ax{r}=R_*\}$ is an axisymmetric coordinate sphere. The radius $R_*$ may be taken to satisfy
    \begin{align}
        R_* \leq \max\lc R_g, R_u, \ls 2\max\{C_g, C_u\} \rs^{1/\qj} \rc,
    \end{align}
    where $C_g,C_u,R_g,R_u$ are any constants such that
    \begin{equation}\label{prop:massangularmomentum:estimates}
        \begin{aligned}
            |\gj-\delta|_\delta &\leq C_g |\bar{x}|^{-\qj}, \qquad \forall |\bar{x}|\geq R_g, \\
            |du^i-d\bar{x}^i|_\delta &\leq C_u |\bar{x}|^{-\qj}, \qquad \forall |\bar{x}|\geq R_u,\
        \end{aligned}
    \end{equation}
    holds for the symmetry-adapted asymptotically flat coordinates $\bar{x}^i$ on $\Mj$.
\end{proposition}

\begin{proof}
    Suppose $(M,g,k)$ is asymptotically flat. Fix an end $\Mend$ and let $(\ax x^1,\ax x^2,\ax x^3)$ be the symmetry-adapted asymptotically flat coordinates given by \Cref{lem:axisymmetriccoordinates}. The pulled back coordinates $\bar{x}^i=\ax x^i\circ\pi$ to the Jang surface then define asymptotically flat coordinates on $\End_0\subset\Mj$. We may apply \Cref{lem:chargedharmonic} to obtain charged harmonic functions on $(\Mj,\gj,\Ej)$ such that
    \begin{align}\label{prop:massangularmomentum:eqn1}
        \Deltac u^i = 0, \qquad u^i = \bar{x}^i + O(\bar{r}^{1-\qj}).
    \end{align}
    Suppose the following estimate holds.
    \begin{align}
        |\gj - \delta|_\delta \leq \varepsilon, \qquad |du^i-d\bar{x}^i|_\delta \leq \varepsilon
    \end{align}
    In particular, if $0\leq\varepsilon\leq\frac{1}{2}$, then
    \begin{align}\label{prop:massangularmomentum:eqn2}
        \sum_{i=1}^3 \frac{|\gradj u^i|_{\gj}}{|\gradj\bar{r}|_{\gj}} 
        \geq 3\frac{(1-\varepsilon)^{3/2}}{(1+\varepsilon)^{1/2}}
        \geq\frac{1}{2}.
    \end{align}
    The upper bound $\smash{\frac{1}{2}}$ is chosen for convenience since $\smash{\frac{1}{2}}<\varepsilon_0$, where $\varepsilon_0$ is the real solution of $36(1-\varepsilon)^3=1+\varepsilon$. The asymptotic conditions for $\Mj$ and $u^i$ imply that there exist constants $C_g,C_u,R_g,R_u>0$ such that \eqref{prop:massangularmomentum:estimates} holds. Observe that \eqref{prop:massangularmomentum:eqn2} is then satisfied for any $\bar{r}$ such that $\max\{C_g \bar{r}^{-\qj}, C_u \bar{r}^{-\qj}\} \leq \frac{1}{2}$. Hence the arguments above hold on $\{\bar{r}\geq R_*\}$ for some $\smash{R_*\leq\max\{ R_g,R_u,\ls 2\max\{C_g,C_u\}\rs^{1/\qj}\}}$. By \Cref{lem:geometricangularmomentum},
    \begin{align}\label{prop:massangularmomentum:eqn3}
        \sum_{i=1}^3 \int_{\{\bar{r}\geq R_*\}} |\gradj u^i|_{\gj}(|h-k|_{\gj}^2 + |X|_{\gj}^2) \, d\Vj 
        \geq \int_{\{\bar{r}\geq R_*\}} \frac{1}{2}|\gradj\bar{r}|_{\gj}(|h-k|_{\gj}^2 + |X|_{\gj}^2) \, d\Vj
        \geq \frac{4\pi\J^2}{\mathcal{R}_\eta(S_{R_*})^3}.
    \end{align}
    Now take the charged harmonic functions $u^i$ in the proof of \Cref{thm:masscharge} to be given by \eqref{prop:massangularmomentum:eqn1} and set $\varepsilon=\frac{3}{5}$ in the subsequent integral inequality. Combined with \eqref{prop:massangularmomentum:eqn3}, we obtain the desired estimate.
    \begin{align}
        3\lp\jang C\cdot\jang E-\frac{7}{10}|\jang\Q| \rp
        &\geq \sum_{i=1}^3 \frac{1}{16\pi}\int_{\Mj} \frac{3}{5}\lp\frac{|\gradc^2 u^i|^2}{|\gradj u^i|}+|\gradj u^i||\Ej|^2\rp + \frac{1}{3}|\gradj u^i|\lp|h-k|^2+|X|^2\rp \, d\Vj \\
        &\geq \sum_{i=1}^3 \frac{3}{80\pi}\int_{\Mj} \frac{|\gradc^2 u^i|^2}{|\gradj u^i|}+|\gradj u^i||\Ej|^2 \, d\Vj + \frac{\J^2}{12\mathcal{R}_\eta(S_{R_*})^3}
    \end{align}

    Suppose $(M,g,k)$ is asymptotically hyperboloidal. Then the asymptotically flat coordinates $\smash{(\bar{x}^1,\bar{x}^2,\bar{x}^3)}$ on $\End_0$ are precisely the ones used by \cite[Corollary 6.11]{sakovich2021jangequationpositivemass} to compute the ADM energy of the end $\Mend$ with coordinates $(\ax{r},\ax\theta,\ax\phi)$. Thus the proof of \Cref{thm:masscharge} may be carried out in the symmetry-adapted coordinates without modification. The rest of the argument then follows as above.
\end{proof}

\section{Area estimate via the charged Hessian}\label{sec:area}

Let $u^i:\Mj\to\R$ be the charged harmonic functions given by \Cref{lem:chargedharmonic} such that $u^i\to x^i$ along $\End_0$, where $x^i$ are asymptotically flat coordinates. Consider the Gram matrix $\G$ associated to $\gradj u^1,\gradj u^2, \gradj u^3$.
\begin{align}
    \G_{ij}=\langle\gradj u^i,\gradj u^j\rangle
\end{align}
Whenever $\G$ is positive-semidefinite, let $\G^{1/2}$ be its square root and take $\sigma_1,\sigma_2,\sigma_3$ to be the nonnegative eigenvalues of $\G^{1/2}$. Define the \emph{singular-value defect} of $U=(u^1,u^2,u^3)$ by
\begin{align}\label{def:phi}
    \Phi = \lVert\G^{1/2}-I\rVert_F
    = \lp \sum_{i=1}^3 (\sigma_i-1)^2 \rp^{1/2},
\end{align}
where $\norm{\cdot}_F$ is the Frobenius norm for matrices. Observe that $\sigma_i$ are the singular values of $dU$. Hence $\Phi$ measures the deviation of $U$ from a local isometry. The asymptotics of $u^i$ along the flat and cylindrical ends (\Cref{lem:chargedharmonic}) imply
\begin{equation}
    \begin{aligned}
        \Phi\to 0 \quad &\text{as} \quad |x|\to\infty \quad \text{along} \quad \End_0, \\
        \Phi\to \sqrt{3} \quad &\text{as} \quad |t|\to\infty \quad \text{along} \quad \End_l.
    \end{aligned}
\end{equation}
For $0\leq a<b$, we shall denote
\begin{align}
    \Sigma_b = \Phi^{-1}(b), \qquad\Omega_a^b = \Phi^{-1}([a,b]), \qquad 
    \reg{[a,b]}{\Phi} = \{t\in[a,b] \,:\, t \text{ is a regular value of $\Phi$}\}.
\end{align}

The definition \eqref{def:phi} is refines the function $F=\norm{\G-I}_F^2$ studied in \cite{allen2025proofspacetimepenroseinequality}. In \Cref{lem:defectproperties}, we establish that $\Phi$ is well-defined on $\{\Phi<1\}$ and satisfies the same basic properties as $F$. In what follows, we shall assume that $\Phi$ is non-constant on $\{0<\Phi<b\}$. The following is the main result of the section.

\begin{theorem}\label{thm:area}
    Let $b\in\reg{(0,\sqrt{2/3})}{\Phi}$. Then for any $\eta\in(0,1)$ and regular values $0<t_-<t_+<b$ of $\Phi$, there exists $t_0\in\reg{[t_-,t_+]}{\Phi}$ and a constant $\mathbf{C}=\mathbf{C}(b,\eta,t_\pm)>0$ such that
    \begin{align}
        |\Sigma_{t_0}| \leq \mathbf{C} \lp \sum_{i=1}^3 \int_{\Omega_0^b} \frac{|\gradc^2 u^i|^2}{|\gradj u^i|} + |\gradj u^i||\Ej|^2 \, dV_{\gj} \rp^2,
    \end{align}
    where $\mathbf{C}$ is explicitly given in \eqref{def:constant}.
\end{theorem}

Throughout this section we shall assume
\begin{align}
    \inf_{t\in(0,1)} |\Sigma_t| > 0,
\end{align}
since the result is immediate otherwise. The main difficulty is establishing a relationship between the Dirichlet energy of $\Phi$ and the volume of its sublevel sets (\Cref{lem:volumeenergy,lem:polyaszego}). This is accomplished by proving a P\'olya-Szeg\H{o} inequality for the symmetric rearrangement of $\Phi$. Our first step is to establish an isoperimetric inequality for the sublevel sets of $\Phi$. Let $\omega_3 = \Vol_{\R^3}(B_1)$.

\begin{lemma}[Isoperimetric inequality]\label{lem:isoperimetric}
    For any $b\in\reg{(0,\sqrt{2/3})}{\Phi}$ and $\eta\in(0,1)$, there exists $b_0\in\reg{(0,b]}{\Phi}$ such that
    \begin{align}
        \Vol(\Omega_{b_0}^t)^{2/3} &\leq C_\eta \bar K(t)|\Sigma_t|\quad\qquad \forall t\in\reg{[b_0,b]}{\Phi}, \\
        \Vol(\Omega_t^{b_0})^{2/3} &\leq C_\eta\bun{K}(t;b)|\Sigma_t|\qquad \forall t\in\reg{(0,b_0]}{\Phi},
    \end{align}
    where $C_\eta = \frac{2+\eta}{3\omega_3^{1/3}}$, $\bar{K}(t) = \lp\frac{1+t/\sqrt{2}}{1-t/\sqrt{3}}\rp^2$, and $\bun{K}(t;b) = \lp\frac{1+t/\sqrt{2}}{1-b/\sqrt{3}}\rp^2$.
\end{lemma}
\begin{proof}
    Fix $\eta\in(0,1)$ and let $b_0\in\reg{(0,b]}{\Phi}$ such that
    \begin{align}
        p(b_0)|\Sigma_{b_0}| \leq (1+\eta) \inf_{t\in\reg{(0,b]}{\Phi}} p(t)|\Sigma_t|,
    \end{align}
    where $p(t)$ is given in \Cref{lem:jacobian}. Suppose $t\in[b_0,b]_\text{reg}$. Let $\smash{J(U)=\sqrt{\det{\G_{ij}}}}$ be the Jacobian of $U$ and denote by $J(U\vert_{\Sigma})$ the Jacobian of the restriction of $U$ to $\Sigma\subset\Mj$. Then we may apply \Cref{lem:jacobian} to obtain
    \begin{align}\label{lem:isoperimetric:eqn1}
        \begin{gathered}
             \int_{\del\Omega_{b_0}^t} J(U\vert_{\del\Omega_{b_0}^t}) \,dA_{\gj}
            \leq p(b_0) |\Sigma_{b_0}| + p(t) |\Sigma_t|
            \leq (2+\eta)p(t)|\Sigma_t|, \\
            \int_{\Omega_{b_0}^t} J(U) \,dV_{\gj} 
            \geq q(t) \Vol(\Omega_{b_0}^t),
        \end{gathered}
    \end{align}
    where the last inequality uses that $q(t)$ is decreasing in $t$. Define $\chi(x) = \mathcal{H}^0(U^{-1}(x)\cap \Omega_{b_0}^t)$ where $\mathcal{H}^0$ is the 0-dimensional Hausdorff measure. Then $\chi$ is non-negative, integer-valued, and BV \cite[Lemma A.4 \& A.5]{allen2025proofspacetimepenroseinequality}. The area formula for $U$ gives
    \begin{align}
        \int_{\Omega_{b_0}^t} J(U) \, dV_{\gj}
        = \int_{\R^3} \chi \, dV_{\R^3}
        = \int_{\R^3} \sum_{n=1}^\infty \mathds{1}_{\{\chi \geq n\}} \, dV_{\R^3}
        = \sum_{n=1}^\infty \Vol_{\R^3}\{\chi \geq n\}.
    \end{align}
    By the subadditivity of $s\mapsto s^{2/3}$,
    \begin{equation}
        \begin{aligned}
            \lp\sum_{n=1}^\infty \Vol_{\R^3}\{\chi \geq n\} \rp^{2/3}
            &\leq \sum_{n=1}^\infty \lp \Vol_{\R^3} \{\chi \geq n\} \rp^{2/3} \\
            &\leq \sum_{n=1}^\infty \frac{|\del\{\chi\geq n\}|}{3\omega_3^{1/3}} && \text{(Isoperimetric inequality)} \\
            &\leq \frac{1}{3\omega_3^{1/3}} \int_{\del\Omega_{b_0}^t} J(U\vert_{\del\Omega_{b_0}^t}) \, dA_{\gj} && \text{(\cite[Eqn. 4.15]{allen2025proofspacetimepenroseinequality})}.
        \end{aligned}
    \end{equation}
    Combined with \eqref{lem:isoperimetric:eqn1}, the desired result follows. The proof when $t\in(0,b_0]_\text{reg}$ follows in a similar manner, with the second inequality in \eqref{lem:isoperimetric:eqn1} replaced by
    \begin{align}
        \begin{gathered}
            \int_{\Omega_t^{b_0}} J(U) \,dV_{\gj} 
            \geq q(b_0) \Vol(\Omega_t^{b_0})
            \geq q(b) \Vol(\Omega_t^{b_0}),
        \end{gathered}
    \end{align}
    again using that $q(t)$ is decreasing in $t$.
\end{proof}

Next we compute an initial upper bound for $|\gradj\Phi|^2$. 

\begin{lemma}\label{lem:gradphi}
    The following holds pointwise on $\{0<\Phi<1\}$:
    \begin{align}
        |\gradj \Phi|^2 \leq \gamma(\Phi)^2\sum_{i=1}^3 \frac{|\gradc^2 u^i|^2}{|\gradj u^i|} + |\gradj u^i||\Ej|^2,
    \end{align}
    where
    \begin{align}
        \gamma(t) = \underset{|\rho|\leq\sqrt{3}}{\max} \ls \frac{2}{3} + \frac{2t}{9} \lp \rho+\sqrt{2(3-\rho^2)} \rp + \frac{(\rho+t)^2}{3+t\rho}\rs^{1/2}.
    \end{align}
\end{lemma}
\begin{proof}
    Let $A=\G^{1/2}$. Since $A$ is symmetric, the definition of the Frobenius norm yields
    \begin{align}
        \Phi^2 = \norm{A-I}_F^2
        = \tr\lp(A-I)(A-I)^T\rp
        = \tr{\G}-2\tr{A}+3.
    \end{align}
    Take $\G(s)$ to be a smooth matrix curve. Then
    \begin{align}
        \frac{d}{ds}\G = \frac{d}{ds}A^2 = \dot{A}A+A\dot{A}.
    \end{align}
    Multiplying by $A^{-1}$ and tracing yields
    \begin{align}
        \tr(A^{-1}\dot{\G}) = \tr(A^{-1}\dot{A}A+\dot{A})=2\tr{\dot{A}}.
    \end{align}
    Hence
    \begin{align}
        \frac{d}{ds}\Phi^2
        = \tr[(I-A^{-1})\dot{\G}]
        = \langle I-A^{-1},\dot{\G}\rangle_F,
    \end{align}
    where $\langle\cdot,\cdot\rangle_F$ is the Frobenius inner product. Next we compute $\dot{\G}$ using the affine derivative $\gradj^{\Ej}$ defined in \eqref{def:chargedconnection}. We have
    \begin{align}
        \dot{\G}_{ij}
        &= \Big\langle\gradj^{\Ej}\langle\gradj u^i,\gradj u^j\rangle,\del_s\Big\rangle
        = \gradc^2 u^i(\gradj u^j,\del_s) + \gradc^2 u^j(\gradj u^i, \del_s) - 2\G_{ij}\langle\Ej,\del_s\rangle,
    \end{align}
    where all inner products are with respect to $\gj$. Thus
    \begin{align}
        \frac{d}{ds}\Phi^2
        = \langle I-A^{-1},\dot{\G}\rangle_F
        &= 2 \sum_{i,j=1}^3 (I-A^{-1})_{ij} \lp \gradc^2 u^i(\gradj u^j,\del_s) - \G_{ij}\langle\Ej,\del_s\rangle \rp.
    \end{align}
    Define $X_i = \sum_{j=1}^3 (I-A^{-1})_{ij} \gradj u^j$. Then the calculations above imply
    \begin{align}
        \gradj\Phi = \frac{1}{\Phi} \lc \sum_{i=1}^3 \gradc^2 u^i(X_i,\cdot) - \langle I-A^{-1},\G\rangle_F \,\Ej \rc.
    \end{align} 
    Since $\gradc^2 u^i$ is symmetric and trace-free, we may apply \Cref{lem:tracelessestimate} to get
    \begin{align}
        |\gradj\Phi|
        \leq \frac{1}{\Phi} \lc \sum_{i=1}^3 \sqrt{\frac{2}{3}} |\gradc^2 u^i||X_i| + \lv\langle I-A^{-1},\G\rangle_F\rv |\Ej| \rc.
    \end{align}
    Square this inequality and apply a weighted Cauchy-Schwarz with the weights
    \begin{equation}
        \begin{aligned}
            a_i &= \frac{2}{3}|\gradj u^i||X_i|^2 \quad \text{for} \quad i=1,2,3, \\
            a_4 &=  \frac{\langle I-A^{-1},\G\rangle_F^2}{\sum_{i=1}^3 |\gradj u^i|},
        \end{aligned}
    \end{equation}
    to obtain the bound
    \begin{equation}
        \begin{aligned}
            |\gradj\Phi|^2
            &\leq \frac{1}{\Phi^2} \lp \sum_{n=1}^4 a_n \rp \lp \sum_{i=1}^3 \frac{2}{3} \frac{|\gradc^2 u^i|^2|X_i|^2}{a_i} + \frac{\langle I-A^{-1},\G\rangle_F^2 |\Ej|^2}{a_4} \rp \\
            &= \frac{1}{\Phi^2} \lp \sum_{n=1}^4 a_n \rp \lp \sum_{i=1}^3 \frac{|\gradc^2 u^i|^2}{|\gradj u^i|} + |\gradj u^i||\Ej|^2 \rp
        \end{aligned}
    \end{equation}
    It remains to estimate the coefficient $\sum_{n=1}^4 a_n$. We compute
    \begin{equation}
        \begin{aligned}
            \sum_{i=1}^3 |X_i|^2
            &= \tr\lp (I-A^{-1})\G(I-A^{-1}) \rp
            = \tr((A-I)^2)
            = \Phi^2 \\
            \langle I-A^{-1},\G\rangle_F^2
            &= \ls \tr(A-I)^2 + \tr(A-I) \rs^2
            = \Phi^2 (\rho+\Phi)^2,
        \end{aligned}
    \end{equation}
    where $\rho$ is defined as in the proof of \Cref{lem:gradu} and satisfies $|\rho|\leq\sqrt{3}$. By \Cref{lem:gradu}, we have
    \begin{align}
        \sum_{n=1}^4 a_n
        \leq \max_{|\rho|\leq\sqrt{3}} \Phi^2 \lc \frac{2}{3}\lp 1 + \frac{\Phi}{3} \ls \rho+\sqrt{2(3-\rho^2)} \rs \rp + \frac{(\rho+\Phi)^2}{3+\rho\Phi} \rc,
    \end{align}
    which implies the result.
\end{proof}

Fix any regular values $0<t_- < t_+ < b$ and let $H:[0,b]\to\R$ satisfy
\begin{align}\label{def:normalization}
    H(0) = 0 \quad\text{and}\quad 
    H'(t) = \frac{C_b(t)}{\gamma(t)}, \quad\text{where}\quad
    C_b(t)=
    \begin{cases}
        \dfrac{\bun{K}(t_-;b)}{\bun{K}(t;b)}, & 0<t<t_- \vspace{0.5em} \\
        1, &t_- \leq t \leq t_+ \vspace{0.5em} \\
        \dfrac{\bar{K}(t_+)}{\bar{K}(t)}, & t_+ < t \leq b.
    \end{cases}
\end{align}
Observe that $H(t)$ is $C^1$ across $t=t_\pm$. Define the normalized defect
\begin{align}
    \tPhi = H(\Phi).
\end{align}
For the remainder of the section, all quantities computed in terms of the normalized defect will be decorated with $\sim$, e.g. $\tSigma_\beta=\tPhi^{-1}(\beta)$ and $\tOmega_\alpha^\beta=\tPhi^{-1}([\alpha,\beta])$. Since $H(t)$ is strictly increasing, the sublevel sets of $\tPhi$ and $\Phi$ are related by
\begin{align}
    \Omega_a^b = \Phi^{-1}([a,b]) = \tPhi^{-1}([H(a),H(b)]) = \tOmega_{H(a)}^{H(b)}.
\end{align}
The following is an immediate consequence of \Cref{lem:gradphi}. 
\begin{corollary}\label{cor:gradtildephi}
    The following holds pointwise on $\{0<\Phi<1\}$:
    \begin{align}
        |\gradj\tPhi|^2 \leq C_b(\Phi)^2 \sum_{i=1}^3 \frac{|\gradc^2 u^i|^2}{|\gradj u^i|} + |\gradj u^i||\Ej|^2.
    \end{align}
\end{corollary}

For any $0<\alpha<\beta<H(1)$, the \emph{increasing distribution function of $\tPhi\vert_{\tOmega_\alpha^\beta}$} is given by
\begin{align}
    \V_\alpha:[\alpha,\beta]\to[0,\infty), \qquad \V_\alpha(\tau) = \Vol_{\gj}(\tOmega_\alpha^\tau).
\end{align}
Note that we have restricted the domain of $\tPhi$ in the above definition only to guarantee that it is well-defined and sufficiently regular (see \Cref{lem:defectproperties}). The \emph{increasing rearrangement of $\tPhi$} is then defined as the increasing distribution function of $\V_\alpha$.

\begin{lemma}[Increasing distribution of $\tPhi$]\label{lem:increasingdistributionproperties}
    For any regular values $0<\alpha<\beta<H(1)$ of $\tPhi$,
    \begin{enumerate}
        \item $\V_\alpha(\tau)$ is increasing in $\tau$ with $\V_\alpha(\alpha) = 0$ and $\V_\alpha(\beta) = \Vol_{\gj}(\tOmega_\alpha^\beta)$,
        \item $\V_\alpha(\tau)$ is absolutely continuous and differentiable almost everywhere, with
        \begin{align*}
            \frac{d\V_\alpha}{d\tau} = \int_{\tSigma_\tau} \frac{dA_{\gj}}{|\gradj \tPhi|}.
        \end{align*}
    \end{enumerate}
\end{lemma}
\begin{lemma}[Increasing rearrangement of $\tPhi$]\label{lem:increasingrearrangementproperties}
    For any regular values $0<\alpha<\beta<H(1)$ of $\tPhi$,
    \begin{enumerate}
        \item $\tPhi_*(v)$ is increasing in $v$ with $\tPhi_*(0)=\alpha$ and $\tPhi_*(\V_\alpha(\beta)) = \beta$,
        \item $\tPhi_*(v) = \min\{\tau\in[\alpha,\beta]\,:\,\V_\alpha(\tau)\geq v\}$ for all $0\leq v\leq\V_\alpha(\beta)$,
        \item $\tPhi_*(\V_\alpha(\tau))=\tau$ and $\V_\alpha(\tPhi_*(v))\leq v$.
    \end{enumerate}
\end{lemma}
Lastly, the \emph{radially increasing rearrangement of $\tPhi$} is given by
\begin{align}
    \tPhi_\star(x) = \tPhi_*(\omega_3 |x|^3), \quad x\in\R^3.
\end{align}
Observe that $\tPhi_\star$ is radially increasing and equidistributed with $\tPhi$. That is, the $\tau$-sublevel set of $\tPhi_\star$ is a Euclidean ball of radius $R(\tau)$ with the same volume as the $\tau$-sublevel set of $\tPhi$. 
\begin{align}
    \Vol_{\R^3}\{x\in\R^3\,:\,\tPhi_\star(x)<\tau\} = \Vol_{\R^3}(B_{R(\tau)}) = \Vol_{\gj}(\tOmega_\alpha^\tau)
\end{align}

The \emph{decreasing distribution function of $\tPhi$} may be analogously defined by
\begin{align}
    \V^\beta:[\alpha,\beta]\to[0,\infty), \quad \V^\beta(\tau) = \Vol_{\gj}(\tOmega_\tau^\beta).
\end{align}
Then the \emph{decreasing rearrangement} $\tPhi^*$ is given by the decreasing distribution function of $\V^\beta$ and the \emph{radially decreasing rearrangement} is given by 
\begin{align}
    \tPhi^\star(x)=\tPhi^*(\omega_3|x|^3), \quad x\in\R^3.
\end{align}
The decreasing rearrangements satisfy corresponding versions of the properties stated in \Cref{lem:increasingdistributionproperties,lem:increasingrearrangementproperties}.

\begin{lemma}\label{lem:volumeenergy}
    Let $0<\alpha<\beta<H(1)$ be regular values of $\tPhi$. Then for any $\tau\in\reg{(\alpha,\beta)}{\tPhi}$, the radially increasing and decreasing rearrangements of $\tPhi$ on $\tOmega_\alpha^\beta$ satisfy
    \begin{align}
        \Vol_{\gj}(\tOmega_\alpha^\tau)^\frac{1}{3} 
        &\leq \frac{1}{3\omega_3^{2/3}(\beta-\tau)^2} \int_{\{\tau\leq\tPhi_\star\leq \beta\}} |\grad \tPhi_\star|^2 \, dV_{\R^3}, \\
        \Vol_{\gj}(\tOmega_\tau^\beta)^\frac{1}{3} 
        &\leq \frac{1}{3\omega_3^{2/3}(\tau-\alpha)^2} \int_{\{\alpha\leq\tPhi^\star\leq\tau\}} |\grad\tPhi^\star|^2 \, dV_{\R^3}.
    \end{align}
\end{lemma}
\begin{proof}
    By the definition of $\tPhi_\star$,
    \begin{align}
        |\grad\tPhi_\star(x)| = 3\omega_3 |x|^2 \lv\frac{d\tPhi_*}{dv}\lp\omega_3|x|^3\rp\rv.
    \end{align}
    After changing to polar coordinates,
    \begin{align}
        \int_{\{\tau\leq\tPhi_\star\leq \beta\}} |\grad \tPhi_\star|^2 \, dV_{\R^3}
        &= \int_{R(\tau)}^{R(\beta)} \lp \int_{S_r} (3\omega_3)^2 r^4 \lv \frac{d\tPhi_*}{dv} \lp\omega_3 r^3\rp \rv^2 d\sigma_r \rp dr \\
        &= (3\omega_3)^3 \int_{R(\tau)}^{R(\beta)} r^6 \lv \frac{d\tPhi_*}{dv}\lp\omega_3 r^3\rp \rv^2 dr,
    \end{align}
    where $d\sigma_r$ is the area element of $S_r\subset\R^3$. By \Cref{lem:increasingdistributionproperties} (2), the following holds almost everywhere.
    \begin{align}
        \frac{d\tPhi_*}{dv}\lp\V_\alpha(\tau)\rp = \frac{1}{\V_\alpha'(\tau)}
    \end{align}
    Since $\V_a(\tau) = \omega_3 R(\tau)^3$, we may change variables to obtain
    \begin{align}\label{lem:volumeenergy:eqn1}
        \int_{\{\tau\leq\tPhi_\star\leq \beta\}} |\grad \tPhi_\star|^2 \, dV_{\R^3}
        = 3^2\omega_3^{2/3} \int_\tau^\beta \frac{\V_\alpha(\tau)^{4/3}}{\V_\alpha'(\tau)} \, d\tau.
    \end{align}
    Then by H\"older's inequality,
    \begin{equation}
        \begin{aligned}
            (\beta-\tau)^2
            &\leq \lp \int_\tau^\beta \frac{\V_\alpha'(\tau)}{\V_\alpha(\tau)^{4/3}} \, d\tau \rp \lp \int_\tau^\beta \frac{\V_\alpha(\tau)^{4/3}}{\V_\alpha'(\tau)} \, d\tau \rp \\
            &= 3\lp \frac{1}{\V_\alpha(\tau)^{1/3}} - \frac{1}{\V_\alpha(\beta)^{1/3}} \rp \int_\tau^\beta \frac{\V_\alpha(\tau)^{4/3}}{\V_\alpha'(\tau)} \, d\tau \\
            &\leq \frac{3}{\Vol_{\gj}(\tOmega_\alpha^\tau)^{1/3}} \int_\tau^\beta \frac{\V_\alpha(\tau)^{4/3}}{\V_\alpha'(\tau)} \, d\tau.
        \end{aligned}
    \end{equation}
    Combining this with \eqref{lem:volumeenergy:eqn1} gives the result. The proof for $\tPhi^\star$ follows in the same manner.
\end{proof}

\begin{lemma}[P\'olya-Szeg\H{o} inequality]\label{lem:polyaszego}
    Suppose $0<a<b<\sqrt{2/3}$ are regular values of $\Phi$ and $\eta\in(0,1)$. Take $b_0\in(0,b]$ to be the regular value of $\Phi$ given by \Cref{lem:isoperimetric}. 
    \begin{enumerate}
        \item If $a \geq b_0$, then the radially increasing rearrangement of $\tPhi$ on $\Omega_a^b$ satisfies
        \begin{align*}
            \int_{\{H(t)\leq\tPhi_\star\leq H(b)\}} |\grad \tPhi_\star|^2 \, dV_{\R^3}
            &\leq \lp 3\omega_3^{1/3}C_\eta\rp^2 \int_{\Omega_t^b} \bar{K}(\Phi)^2 |\gradj\tPhi|^2 \, dV_{\gj}, \qquad \forall t\in\reg{[a,b]}{\Phi}.
        \end{align*}

        \item If $a\leq b_0$, then the radially decreasing rearrangement of $\tPhi$ on $\Omega_a^b$ satisfies
        \begin{align*}
            \int_{\{H(a)\leq\tPhi^\star\leq H(t)\}} |\grad \tPhi^\star|^2 \, dV_{\R^3} 
            &\leq \lp 3\omega_3^{1/3}C_\eta\rp^2 \int_{\tOmega_a^t} \bun{K}(\Phi;b)^2 |\gradj\tPhi|^2 \, dV_{\gj}, \qquad \forall t\in\reg{[a,b]}{\Phi}.
        \end{align*}
    \end{enumerate}
\end{lemma}
\begin{proof}
    Suppose $a \geq b_0$ and let $t\in\reg{[a,b]}{\Phi}$. Set
    \begin{align}
        \alpha = H(a), \quad \beta=H(b), \quad \tau=H(t).
    \end{align}
    Since $\V_\alpha(\tau) = \Vol_{\gj}(\Omega_a^t) \leq \Vol_{\gj}(\Omega_{b_0}^t)$, we may apply \Cref{lem:isoperimetric} to \eqref{lem:volumeenergy:eqn1} and obtain
    \begin{align}
        \int_{\{\tau\leq\tPhi_\star\leq\beta\}} |\grad\tPhi_\star|^2 \, dV_{\R^3}
        = 3^2\omega_3^{2/3} \int_\tau^\beta \frac{\V_\alpha(\tau)^{4/3}}{\V_\alpha'(\tau)} \, d\tau
        \leq \lp 3\omega_3^{1/3}C_\eta\rp^2 \int_\tau^\beta \bar{K}(H^{-1}(\tau))^2 \frac{|\tSigma_\tau|^2}{\V_\alpha'(\tau)} \, d\tau.
    \end{align}
    For almost every $\tau\in[\alpha,\beta]$, we have by \Cref{lem:increasingdistributionproperties} (2)
    \begin{align}
        |\tSigma_\tau|^2 
        \leq \lp \int_{\tSigma_\tau} \frac{dA_{\gj}}{|\gradj\tPhi|} \rp \lp \int_{\tSigma_\tau} |\gradj\tPhi| \, dA_{\gj} \rp
        = \V_\alpha'(\tau) \int_{\tSigma_\tau} |\gradj\tPhi| \, dA_{\gj}.
    \end{align}
    Combining this with the previous inequality yields
    \begin{align}
        \int_{\{\tau\leq\tPhi_\star\leq\beta\}} |\grad\tPhi_\star|^2 \, dV_{\R^3}
        \leq \lp 3\omega_3^{1/3}C_\eta\rp^2 \int_{\tOmega_\tau^\beta} \bar{K}(\Phi)^2 |\gradj\tPhi|^2 \, dV_{\gj},
    \end{align}
    as desired. The proof when $a\leq b_0$ follows in an identical manner.
\end{proof}

\begin{proof}[Proof of \Cref{thm:area}]
    Let $b\in\reg{(0,\sqrt{2/3})}{\Phi}$, $\eta\in(0,1)$, and take $b_0\in(0,b]$ to be the regular value of $\Phi$ given by \Cref{lem:isoperimetric}. For regular values $0<t_-<t_+<b$ of $\Phi$, take $\tPhi$ to be the normalized defect defined in \eqref{def:normalization}. We denote
    \begin{align}
        \beta=H(b), \quad \beta_0=H(b_0), \quad \tau_+=H(t_+), \quad \tau_-=H(t_-),
    \end{align}
    and introduce the effective lengths
    \begin{align}\label{def:lengths}
        I_- = \frac{\tau_-}{\bun{K}(t_-;b)}, \qquad
        I_0 = \tau_+ - \tau_-, \qquad
        I_+ = \frac{\beta - \tau_+}{\bar{K}(t_+)}.
    \end{align}
    Fix any $\tau_0\in[\tau_-,\tau_+]$ such that
    \begin{align}\label{thm:area:eqn1}
        |\tSigma_{\tau_0}|
        \leq \frac{1}{I_0} \int_{\tau_-}^{\tau_+} |\tPhi^{-1}(\tau)|\,d\tau
        \leq \frac{1}{I_0} \Vol_{\gj}(\tOmega_{\tau_-}^{\tau_+})^{1/2}\lp \int_{\tOmega_{\tau_-}^{\tau_+}} |\gradj\tPhi|^2 \, dV_{\gj} \rp^{1/2},
    \end{align}
    where the second inequality follows by the co-area formula and H\"older's inequality. We proceed by cases depending on where $\beta_0$ lies with respect to the partition $0<\tau_-<\tau_+<\beta$. 

    \emph{\underline{Case 1}: $\tau_+\leq\beta_0<\beta$.} Fix any $0<\varepsilon<\tau_-$ and take $\tPhi^\star$ to be the radially decreasing rearrangement of $\tPhi$ on $\tOmega_\varepsilon^{\beta_0}$. By the results of this section,
    \begin{equation}
    \begin{aligned}
        \Vol_{\gj}(\tOmega_{\tau_-}^{\tau_+})^{1/3}
        \leq \Vol_{\gj}(\tOmega_{\tau_-}^{\beta_0})^{1/3}
        &\leq \frac{1}{3\omega_3^{2/3}(\tau_--\varepsilon)^2} \int_{\{\varepsilon\leq\tPhi^\star\leq\tau_-\}} |\grad\tPhi^\star|^2 \, dV_{\R^3} && \text{(Lem.~\ref{lem:volumeenergy})} \\
        &\leq \frac{3C_\eta^2}{(\tau_--\varepsilon)^2} \int_{\tOmega_\varepsilon^{\tau_-}} \bun{K}(\Phi;b)^2|\gradj\Phi|^2 \, dV_{\gj} && \text{(Lem.~\ref{lem:polyaszego})} \\
        &\leq 3C_\eta^2\ls\frac{\bun{K}(t_-;b)}{\tau_--\varepsilon}\rs^2 \I(\tOmega_\varepsilon^{\tau_-}), && \text{(Cor.~\ref{cor:gradtildephi})}
    \end{aligned}
    \end{equation}
    where $\I(\Omega)$ is the integral quantity defined for any $\Omega\subset\Mj$ by
    \begin{align}
        \I(\Omega) \equiv \sum_{i=1}^3 \int_\Omega \frac{|\gradc^2 u^i|^2}{|\gradj u^i|}+|\gradj u^i||\Ej|^2 \, d\mu_\Omega.
    \end{align}
    
    \emph{\underline{Case 2}: $0<\beta_0\leq\tau_-$.} Let $\tPhi_\star$ be the radially increasing rearrangement of $\tPhi$ on $\tOmega_{\beta_0}^\beta$. Then
    \begin{equation}
    \begin{aligned}
        \Vol_{\gj}(\tOmega_{\tau_-}^{\tau_+})^{1/3}
        \leq \Vol_{\gj}(\tOmega_{\beta_0}^{\tau_+})^{1/3}
        &\leq \frac{1}{3\omega_3^{2/3}(\beta-\tau_+)^2} \int_{\{\tau_+\leq\tPhi_\star\leq\beta\}} |\grad\tPhi_\star|^2 \, dV_{\R^3} && \text{(Lem.~\ref{lem:volumeenergy})} \\
        &\leq \frac{3C_\eta^2}{(\beta-\tau_+)^2} \int_{\tOmega_{\tau_+}^\beta} \bar{K}(\Phi)^2 |\gradj\tPhi|^2 \, dV_{\gj} && \text{(Lem.~\ref{lem:polyaszego})} \\
        &\leq 3C_\eta^2\ls\frac{\bar{K}(t_+)}{\beta-\tau_+}\rs^2 \I(\tOmega_{\tau_+}^\beta). && \text{(Cor.~\ref{cor:gradtildephi})}
    \end{aligned}
    \end{equation}

    \emph{\underline{Case 3}: $\tau_-\leq \beta_0\leq\tau_+$}. By subadditivity,
    \begin{align}
        \Vol_{\gj}(\tOmega_{\tau_-}^{\tau_+})^{1/3}
        \leq \Vol_{\gj}(\tOmega_{\tau_-}^{\beta_0})^{1/3} + \Vol_{\gj}(\tOmega_{\beta_0}^{\tau_+})^{1/3}.
    \end{align}
    The summands may then be treated by the previous cases.
    \begin{align}
        \Vol_{\gj}(\tOmega_{\tau_-}^{\tau_+})^{1/3}
        &\leq 3C_\eta^2 \max\lc\frac{\bun{K}(t_-;b)}{\tau_--\varepsilon},\frac{\bar{K}(t_+)}{\beta-\tau_+} \rc^2 \I(\tOmega_\varepsilon^{\tau_-}\cup\tOmega_{\tau_+}^\beta)
    \end{align}
    Since this encompasses the bounds from the previous cases, we may use it to estimate \eqref{thm:area:eqn1} from above
    \begin{align}
        |\tSigma_{\tau_0}|
        \leq \frac{3^{3/2}C_\eta^3}{I_0} \max\lc\frac{\bun{K}(t_-;b)}{\tau_--\varepsilon},\frac{\bar{K}(t_+)}{\beta-\tau_+}\rc^3 \I(\tOmega_\varepsilon^{\tau_-}\cup\tOmega_{\tau_+}^\beta)^{3/2} \I(\tOmega_{\tau_-}^{\tau_+})^{1/2},
    \end{align}
    where we have again used \Cref{cor:gradtildephi}. By the elementary inequality $m^{3/2}n^{1/2} \leq \frac{3^{3/2}}{16}(m+n)^2$,
    \begin{align}
        |\tSigma_{\tau_0}|
        \leq \frac{27C_\eta^3}{16 I_0} \max\lc\frac{\bun{K}(t_-;b)}{\tau_--\varepsilon},\frac{\bar{K}(t_+)}{\beta-\tau_+}\rc^3 \I(\tOmega_\varepsilon^\beta)^2.
    \end{align}
    The result follows by taking the limit $\varepsilon\to0$ via Fatou's lemma to obtain
    \begin{align}\label{def:constant}
        |\tSigma_{\tau_0}|
        \leq \underbrace{\frac{27C_\eta^3}{16 I_0} \max\lc \frac{1}{I_+},\frac{1}{I_-} \rc^3}_{\text{\normalsize$\equiv \mathbf{C}(b,\eta,t_\pm)$}} \I(\tOmega_0^\beta)^2.
    \end{align}
\end{proof}

\begin{remark}\label{remark:parameters}
    It is clear from the definition of $C_\eta$ that $\mathbf{C}$ is minimized in the limit $\eta\to0$.
    \begin{align}
        \mathbf{C}(b,0,t_\pm) = \frac{1}{2\omega_3I_0} \max\lc\frac{1}{I_+},\frac{1}{I_-}\rc^3
    \end{align}
   Then unpack $I_0,I_\pm$ using the definition of $H$ to obtain
    \begin{align}\label{eqn:lengths}
        I_0 = \int_{t_-}^{t_+} \frac{dt}{\gamma(t)}, \qquad
        I_+ = \int_{t_+}^b \frac{dt}{\bar{K}(t)\gamma(t)}, \qquad
        I_- = \int_0^{t_-} \frac{dt}{\bun{K}(t;b) \gamma(t)}.
    \end{align}
    In \Cref{appendix:constant:gamma}, we observe that $\gamma(t)$ admits an upper bound that is linear in $t$.
    \begin{align}
        \gamma(t) \leq \sqrt{\frac{5}{3}}\lp \frac{3}{10}t + 1 \rp
    \end{align}
    Hence in \Cref{lem:gradphi} and all subsequent definitions involving $\gamma(t)$, we may replace $\gamma(t)$ by this linear function and directly integrate \eqref{eqn:lengths} to obtain
    \begin{equation}
    \begin{aligned}
        I_0 &= 2\sqrt{\frac{5}{3}} \ln\lp\frac{3t_++10}{3t_-+10}\rp, \\
        I_+ &= 4\sqrt{\frac{5}{3}} \lc \frac{C_+}{3}\ln\lp\frac{3b+10}{3t_++10}\rp + D_+ \ln\lp\frac{b+\sqrt{2}}{t_++\sqrt{2}}\rp + \frac{E_+(b-t_+)}{(b+\sqrt{2})(t_++\sqrt{2})} \rc, \\
        I_- &= 4\sqrt{15} \lp1-\frac{b}{\sqrt{3}}\rp^2 \lc \frac{C_-}{3} \ln\lp\frac{3t_-}{10}+1\rp + D_- \ln\lp\frac{t_-}{\sqrt{2}}+1\rp + \frac{E_-t_-}{\sqrt{2}(t_-+\sqrt{2})} \rc,
    \end{aligned}
    \end{equation}
    where $C_\pm,D_\pm,E_\pm$ are explicit constants given in \Cref{appendix:constant:lengths}. From here, we use Mathematica to approximate the optimal parameter values.
    \begin{align}
        b \approx 0.62067668, \quad
        t_+ \approx 0.34601268, \quad
        t_- \approx 0.20277676
        \implies 
        \mathbf{C} \approx 7006.35698657
    \end{align}
\end{remark}

\section{Proof of the main results}
\begin{proof}[Proof of \Cref{thm:main_AF,thm:main_AH}]
    Let $(M,g,k,\E)$ be a charged initial data set which is either asymptotically flat or asymptotically hyperboloidal. Fix an end $\Mend\subset M$ with energy $E=E_\text{ADM}$ or $E_\text{hyp}$. By \Cref{lem:generalizedexteriorregion}, \Cref{prop:jang}, and \Cref{prop:jangelectricfield}, there is a distinguished outermost apparent horizon $\Sigma$ with respect to $\Mend$ and a charged Jang surface $(\Mj,\gj,\Ej)$ satisfying the properties in \Cref{table:jang}. By \Cref{prop:masschargeangularmomentum}, there exist asymptotically linear charged harmonic functions $u^1,u^2,u^3$ on $\End_0\subset\Mj$ such that
    \begin{equation}\label{thm:main:eqn1}
        \begin{aligned}
            E-\frac{7}{10}|\Q| - \frac{\J^2}{36\mathcal{R}_\eta(S_{R_*})^3} \geq \frac{1}{80\pi} \sum_{i=1}^3 \int_{\Mj} \frac{|\gradc^2 u^i|^2}{|\gradj u^i|}+|\gradj u^i||\Ej|^2 \, dV_{\gj},
        \end{aligned}
    \end{equation}
    where we have used in the asymptotically hyperboloidal case that $\jang E=2\ax{E}\leq2E$ (\Cref{lem:axisymmetriccoordinates}) and in both cases the fact that $\Qj=\Q$ (\Cref{prop:jangelectricfield}). Let $\Phi$ be the singular-value defect for $u^i$ defined by the formula \eqref{def:phi}. Since $(u^1,u^2,u^3)$ forms an asymptotically flat coordinate chart on $\End_0$, the sublevel set $\{0\leq\Phi<t\}$ contains a possibly smaller asymptotically flat end within $\End_0$ for any $t>0$. On the other hand, $\Phi\to \sqrt{3}$ along each cylindrical end of $\Mj$ by \Cref{lem:chargedharmonic} (3). Hence for any regular value $0<t<\sqrt{3}$, the surface $\Phi^{-1}(t)$ separates the cylindrical ends of $\Mj$ from an asymptotically flat end contained in $\End_0$. Moreover, the area of such a surface satisfies $|\Phi^{-1}(t)|\geq\A_\text{min}(\Sigma)$ \Cref{prop:jang} (2). By the results of \Cref{sec:area}, there is a regular value $t_0\in(0,\sqrt{2/3})$ and an explicit constant $\mathbf{C}$ such that
    \begin{align}\label{thm:main:eqn2}
        \A
        \leq |\Phi^{-1}(t_0)|
        &\leq \mathbf{C} \lp \sum_{i=1}^3 \int_{\Mj} \frac{|\gradc^2 u^i|^2}{|\gradj u^i|} + |\gradj u^i||\Ej|^2 \, dV_{\gj} \rp^2 && \text{(Thm.~\ref{thm:area})} \\
        &\leq \mathbf{C} (80\pi)^2 \lp E-\frac{7}{10}|\Q| - \frac{\J^2}{36\mathcal{R}_\eta(S_{R_*})^3} \rp^2 && \text{(Eqn.~\ref{thm:main:eqn1})}.
    \end{align}
    By \Cref{remark:parameters}, there exist values of $b,\eta,t_\pm$ such that
    \begin{align}
        \frac{7}{10}\geq\frac{1}{40\sqrt{\pi\mathbf{C}}}, \qquad
        \frac{1}{36}\geq\frac{2}{5\sqrt{\mathbf{C}}}.
    \end{align}
    Hence we may take the square root of \eqref{thm:main:eqn2} and rearrange to obtain
    \begin{align}\label{thm:main:eqn3}
        E
        \geq \frac{1}{20\sqrt{\pi\mathbf{C}}}\lp\sqrt{\frac{\A}{16\pi}} + \frac{|\Q|}{2} + \frac{\sqrt{16\pi}\J^2}{2\mathcal{R}_\eta(S_{R_*})^3}\rp.
    \end{align}
    Squaring the expression in the parentheses and strategically omitting terms yields
    \begin{align}
        \lp\sqrt{\frac{\A}{16\pi}} + \frac{|\Q|}{2} + \frac{\sqrt{16\pi}\J^2}{2\mathcal{R}_\eta(S_{R_*})^3}\rp^2
        \geq \lp\sqrt{\frac{\A}{16\pi}} + \frac{|\Q|}{2}\rp^2 + \frac{\J^2\sqrt{\A}}{\mathcal{R}_\eta(S_{R_*})^3}.
    \end{align}
    As discussed in \Cref{sec:intro}, the Penrose inequality is distinct from the positive mass inequality when $\A\geq4\pi\sqrt{\Q^4+4\J^2}$. In particular, this implies $1\geq 2|\Q|\sqrt{\frac{\pi}{\A}}$. Hence
    \begin{align}
        E\geq\frac{1}{20\sqrt{\pi\mathbf{C}}}\ls\lp\sqrt{\frac{\A}{16\pi}} + \Q^2\sqrt{\frac{\pi}{\A}}\rp^2 + \frac{\J^2\sqrt{\A}}{\mathcal{R}_\eta(S_{R_*})^3} \rs.
    \end{align}
    The proof of the inequality is completed by observing that any asymptotically flat initial data sets satisfying \eqref{cond:AF_asymptotics} has trivial linear momentum. The stated value for $\C$ follows from \Cref{remark:parameters}.
\end{proof}

\appendix
\section{Elementary estimate}

\begin{lemma}\label{lem:tracelessestimate}
    Let $T$ be a symmetric, trace-free 2-tensor. Then for any unit vector $e$,
    \begin{align}
        |T(e,\cdot)|^2 \leq \frac{2}{3}|T|^2.
    \end{align}
\end{lemma}
\begin{proof}
    Extend $e$ to an orthonormal basis with $e_1=e$ and express $T$ in this basis as $T=(a_{ij})$. We compute
    \begin{align}
        |T(e,\cdot)|^2
        = a_{11}^2 + a_{12}^2 + a_{13}^2, \qquad
        |T|^2 = a_{11}^2 + a_{22}^2 + a_{33}^2 + 2(a_{12}^2 + a_{13}^2 + a_{23}^2).
    \end{align}
    The trace-free assumption implies
    \begin{align}
        a_{11}^2
        = (a_{22} + a_{33})^2
        \leq 2 (a_{22}^2 + a_{33}^2),
    \end{align}
    from which it follows
    \begin{align}
        |T|^2 
        \geq \frac{3}{2} a_{11}^2 + 2(a_{12}^2 + a_{13}^2 + a_{23}^2)
        \geq \frac{3}{2}|T(e,\cdot)|^2.
    \end{align}
\end{proof}

\section{Properties of the defect}
We establish some basic properties for $\Phi$ on the set where $\Phi<1$.

\begin{lemma}\label{lem:defectproperties}
    For every $b\in(0,1)$:
    \begin{enumerate}
        \item $\Phi$ is well-defined on $\{\Phi<1\}$ and smooth on $\{0<\Phi<1\}$.
        \item The restriction of $\Phi$ to $\{0<\Phi< b\}$ is proper.
        \item The regular values of $\Phi$ have full measure in $(0,b)$.
        \item If $0<t_1<t_2<b$ are regular values, then $\del\Omega_{t_1}^{t_2} = \Sigma_{t_1} \cup \Sigma_{t_2}$.
        \item If $\Phi(x)<1$, then $U=(u^1,u^2,u^3)$ is a local diffeomorphism near $x$.
    \end{enumerate}
\end{lemma}
\begin{proof}
    The definition of $\Phi$ implies the singular values of $dU$ satisfy
    \begin{align}
        |\sigma_i - 1| < 1 \quad\text{on}\quad\{\Phi<1\},
    \end{align}
    so they in particular must be positive. Since $\G= (dU)(dU)^T$, it follows that $\G$ is positive definite on $\{\Phi<1\}$. For smoothness, recall that the map $A\mapsto A^{1/2}$ is smooth on the set of positive definite symmetric matrices and that the Frobenius norm is smooth away from the zero matrix. The other properties follow from \cite[Lemma A.1 \& A.2]{allen2025proofspacetimepenroseinequality}.
\end{proof}

Let $J(U) = \sqrt{\det{\G}}$ be the Jacobian of $U$ and denote by $J(U\vert_\Sigma)$ the Jacobian of the restricted map $U\vert_\Sigma$ to any surface $\Sigma\subset\Mj$. We obtain the following bounds for $\J$ on the set $\{0<\Phi<\sqrt{2/3}\}$.

\begin{lemma}\label{lem:jacobian}
    Let $0<a<t<\sqrt{\frac{2}{3}}$ be regular values of $\Phi$. Then
    \begin{align}
        J(U)&\geq q(t) = \lp 1-\tfrac{t}{\sqrt{3}} \rp^3\quad\text{on}\quad\Omega_a^t, \\
        J(U\vert_{\del\Omega_a^t})&\leq p(t) = \lp1+\tfrac{t}{\sqrt{2}}\rp^2\quad\text{on}\quad\del\Omega_a^t.
    \end{align}
\end{lemma}
\begin{proof}
    The Jacobian of $U$ is given by $J(U)=\sigma_1\sigma_2\sigma_3$ where $\sigma_i$ are the singular values of $dU$. Since
    \begin{align}
        \Phi^2 = \sum_{i=1}^3 (\sigma_i-1)^2 \leq t^2,
    \end{align}
    on the set $\Omega_a^t$, we may reformulate the problem as minimizing the function
    \begin{align}
        P:\R^3\to\R, \qquad P(\sigma)=\sigma_1\sigma_2\sigma_3,
    \end{align}
    over the closed ball $D_t\equiv\overline{B_t((1,1,1))}\subset\R^3$. We claim that the minimum is achieved on the boundary sphere. Otherwise, $P$ satisfies at any interior minimum
    \begin{align}
        \grad P = (\sigma_2\sigma_3,\sigma_1\sigma_3,\sigma_1\sigma_2)=0,
    \end{align}
    so two or more singular values must vanish. But then $\sum_{i=1}^3 (\sigma_i-1)^2 \geq 2$, which is impossible for elements of $D_t$ since $t<\sqrt{2/3}$. Hence any minimizer of $P$ satisfies
    \begin{align}\label{lem:jacobian:lagrangeconstraint}
        \sum_{i=1}^3 (\sigma_i-1)^2 = t^2.
    \end{align}
    By the method of Lagrange multipliers, each minimum of $P$ subject to the constraint \eqref{lem:jacobian:lagrangeconstraint} satisfies 
    \begin{align}
        \begin{gathered}\label{lem:jacobian:lagrangeequations}
            \sigma_2\sigma_3 = 2\lambda(\sigma_1-1) \\
            \sigma_1\sigma_3 = 2\lambda(\sigma_2-1) \\
            \sigma_1\sigma_2 = 2\lambda(\sigma_3-1)
        \end{gathered}
    \end{align}
    for some constant $\lambda$. Subtracting the equations pairwise yields
    \begin{align}\label{lem:jacobian:differenceequations}
        (\sigma_i-\sigma_j)(\sigma_k+2\lambda) = 0,
    \end{align}
    for any $i,j,k$ distinct. Evaluating the constraint equation for $\sigma_1=\sigma_2=\sigma_3$ shows that there are two extremizers of $P$ lying on the diagonal, $\sigma_i^\pm = 1 \pm \frac{t}{\sqrt{3}}$. It is clear from definition that
    \begin{align}
        P(\sigma^+)\geq P(\sigma^-) = \lp 1 - \frac{t}{\sqrt{3}} \rp^3.
    \end{align}
    
    It remains to consider the case when $\sigma$ does not lie on the diagonal. The difference equations \eqref{lem:jacobian:differenceequations} imply the solution must be of the form $\sigma_1=\sigma_2=-2\lambda$, $\sigma_3\neq -2\lambda$. Substituting into \eqref{lem:jacobian:lagrangeequations} gives
    \begin{align}
        4\lambda^2 = 2\lambda(\sigma_3-1),
    \end{align}
    which implies $\sigma_3=2\lambda+1$. The constraint equation then becomes
    \begin{align}
        2(2\lambda+1)^2 + 4\lambda^2 = t^2,
    \end{align}
    which has no real solutions for $t<\sqrt{2/3}$. Hence \eqref{lem:jacobian:lagrangeequations} has no off-diagonal solutions when subject to \eqref{lem:jacobian:lagrangeconstraint} for $t<\sqrt{2/3}$. This concludes the proof for the lower bound of $J(U)$. For the upper bound, suppose $\sigma_1\geq\sigma_2\geq\sigma_3$. The Jacobian restricted to the tangent plane of any surface $\Sigma$ satisfies
    \begin{equation}
    \begin{aligned}
        J(U\vert_\Sigma)
        \leq \sigma_1\sigma_2
        &= [(\sigma_1-1)+1][(\sigma_2-1)+1] \\
        &= \prod_{i=1}^2 (\sigma_i-1) + \sum_{i=1}^2 (\sigma_i-1) + 1 \\
        &\leq \frac{1}{2} \sum_{i=1}^2 (\sigma_i-1)^2  + \lp 2\sum_{i=1}^3 (\sigma_i-1)^2 \rp^{1/2} + 1 \\
        &\leq \frac{1}{2}\Phi^2 + \Phi\sqrt{2} + 1 \\
        &= \lp \frac{\Phi}{\sqrt{2}} + 1 \rp^2.
    \end{aligned}
    \end{equation}
    The desired upper bound follows since $\Phi\leq t$ on $\del\Omega_a^t=\Sigma_a\cup\Sigma_t$.
\end{proof}

\begin{lemma}\label{lem:gradu}
    The following holds pointwise on $\{\Phi<1\}$.
    \begin{align}
        |\gradj u^i| &\leq \underset{|\rho|\leq\sqrt{3}}{\max} \lp 1 + \frac{\Phi}{3} \ls \rho+\sqrt{2(3-\rho^2)} \rs\rp, \quad i=1,2,3 \\
        \sum_{i=1}^3 |\gradj u^i| &\geq \underset{|\rho|\leq\sqrt{3}}{\min} \lp 3 + \rho\Phi \rp
    \end{align}
\end{lemma}
\begin{proof}
    On $\{\Phi<1\}$, the Gram matrix $\G$ admits a square root $\G^{1/2}$ with positive eigenvalues $\sigma_1\geq\sigma_2\geq\sigma_3>0$. Let $\G^{1/2}=Q\,\text{diag}(\sigma_1,\sigma_2,\sigma_3)\,Q^T$ be a diagonalization with $Q$ orthogonal. Observe
    \begin{align}\label{lem:gradu:eqn1}
        |\gradj u^i| 
        = \sqrt{\G_{ii}}
        = \lp\sum_{j=1}^3 Q_{ij}^2 \sigma_j^2\rp^{1/2}
        \leq \max_j \sigma_j \lp\sum_{j=1}^3 Q_{ij}^2\rp^{1/2} 
        = \sigma_1.
    \end{align}
    Let $S \equiv \sum_{i=1}^3 \xi_i \equiv \sum_{i=1}^3 (\sigma_i-1)$ and define $\rho$ by the relation
    \begin{align}
        S=\rho\Phi \iff
        \sum_{i=1}^3 (\sigma_i-1) = \rho \lp \sum_{i=1}^3 (\sigma_i-1)^2 \rp^{1/2}.
    \end{align}
    Observe by Cauchy-Schwarz that
    \begin{align}
        S^2 \leq 3 \sum_{i=1}^3\xi_i^2 = 3\Phi^2,
    \end{align}
    so $\rho$ must lie in the interval $-\sqrt{3}\leq\rho\leq\sqrt{3}$. By Young's inequality,
    \begin{align}
        \Phi^2 - \xi_1^2
        = \xi_2^2 + \xi_3^2
        \geq \frac{1}{2}(\xi_2+\xi_3)^2
        = \frac{1}{2}(S-\xi_1)^2.
    \end{align}
    Solving the quadratic inequality yields
    \begin{align}
        \xi_1 
        \leq \frac{1}{3} \ls S+\sqrt{2(3\Phi^2-S^2)} \rs
        = \frac{\Phi}{3} \ls \rho+\sqrt{2(3-\rho^2)} \rs.
    \end{align}
    Combining with \eqref{lem:gradu:eqn1} and taking the maximum over $\rho$ implies the first inequality in \Cref{lem:gradu}. For the second inequality, use Cauchy-Schwarz to get
    \begin{align}
        \sum_{j=1}^3 Q_{ij}^2 \sigma_j
        \leq \lp \sum_{j=1}^3 Q_{ij}^2 \rp^\frac{1}{2} \lp \sum_{j=1}^3 Q_{ij}^2 \sigma_j^2 \rp^\frac{1}{2}
        = \lp \sum_{j=1}^3 Q_{ij}^2 \sigma_j^2 \rp^\frac{1}{2}
        = |\gradj u^i|.
    \end{align}
    Then summing over $i$ yields
    \begin{align}
        \sum_{i=1}^3 |\gradj u^i|
        \geq \sum_{i=1}^3 \sigma_i
        = 3 + \sum_{i=1}^3 (\sigma_i-1)
        = 3 + S,
    \end{align}
    and the result follows by minimizing over $\rho$.
\end{proof}

\section{Computing the constant}\label{appendix:constant}

\subsection{Upper bound for \texorpdfstring{$\gamma$}{gamma}}\label{appendix:constant:gamma}
Let $\gamma(t)^2=\max_{|\rho|\leq\sqrt{3}}\Gamma_t(\rho)$ be the square of the coefficient obtained in \Cref{lem:gradphi}, where
\begin{align}\label{appendix:constant:gammadef}
    \Gamma_t(\rho) = \frac{2}{3} + \frac{2t}{9} \lp \rho + \sqrt{2(3-\rho^2)} \rp + \frac{(\rho+t)^2}{3+t \rho} .
\end{align}
We claim that
\begin{align}\label{appendix:constant:gammaproperty}
    \Gamma_t(\rho)\geq \Gamma_t(-\rho) \quad\text{for}\quad\rho\geq 0.
\end{align}
Expanding both sides, we may observe that first term in \eqref{appendix:constant:gammadef} remains unchanged while the second term clearly decreases for $-\rho$. It remains to prove
\begin{align}\label{appendix:constant:eqn1}
    \frac{(\rho+t)^2}{3+t\rho} \geq \frac{(t-\rho)^2}{3-t\rho}.
\end{align}
For $0<t\leq 1$ and $0\leq\rho\leq\sqrt{3}$, we obtain
\begin{align}
    (3-t\rho)(\rho + t)^2 -(3+t\rho)(t-\rho)^2
    = 2t\rho\ls6-(\rho^2+t^2)\rs
    \geq 4t\rho
    \geq 0.
\end{align}
Rearranging yields the desired inequality \eqref{appendix:constant:eqn1}. Next, we estimate the middle term in \eqref{appendix:constant:gammadef} using Young's inequality. For any $\varepsilon>0$, 
\begin{align}
    \frac{2\sqrt{2}t}{9}\sqrt{3-\rho^2}
    \leq \frac{2t^2}{81\varepsilon} + \varepsilon(3-\rho^2).
\end{align}
Combined with \eqref{appendix:constant:gammaproperty}, this implies
\begin{align}\label{appendix:constant:eqn2}
    \gamma(t)^2 \leq \max_{0\leq \rho\leq\sqrt{3}} \lp \frac{2}{3} + 3\varepsilon + \frac{2t^2}{81\varepsilon} + \Delta_{t,\varepsilon}(\rho)\rp,
\end{align}
where $\Delta_{t,\varepsilon}(\rho) = -\varepsilon\rho^2 + \frac{2t}{9}\rho + \frac{(\rho+t)^2}{3+t\rho}$. We may obtain a $\rho$-independent estimate for $\gamma$ by choosing $\varepsilon$ so that
\begin{align}\label{appendix:constant:deltacondition}
    \Delta_{t,\varepsilon}(\rho)\leq\Delta_{t,\varepsilon}(\sqrt{3}).
\end{align}
The best constant in \Cref{thm:area} is obtained by maximizing $\varepsilon$ subject to \eqref{appendix:constant:deltacondition}. However, analytically solving the inequality yields a marginal improvement while further complicating the exposition, so we instead take $\varepsilon=\frac{1}{3}$ as a sufficiently close approximation. Let
\begin{align}\label{appendix:constant:eqn3}
    \Delta_t(\rho) 
    = \Delta_{t,\frac{1}{3}}(\rho)
    = -\frac{1}{3}\rho^2 + \frac{2t}{9}\rho + \frac{(\rho+t)^2}{3+t\rho}.
\end{align}
Observe that the first and third terms cancel when evaluated at $t=0$. This motivates the uniform estimate 
\begin{align}
    \Delta_t(\rho)\leq t, \qquad 0\leq\rho\leq\sqrt
3.
\end{align}
Rearranging and multiplying by $9(3+t\rho)$ produces the equivalent inequality
\begin{equation}
    \begin{aligned}\label{appendix:constant:eqn4}
        0 
        &\leq 9(3+t\rho)(t-\Delta_t(\rho)) \\
        &= (3+t\rho)(9t-2t\rho+3\rho^2) - 9(\rho+t)^2 \\
        &= t\ls A(\rho) + tB(\rho) \rs,
    \end{aligned}
\end{equation}
where $A(\rho) = 3\rho^3-24\rho+27$ and $B(\rho) = -2\rho^2+9\rho-9$. Observe that $B(\rho)\geq 0$ on the interval $\frac{3}{2}\leq\rho\leq 3$. Furthermore, since
\begin{align}
    A'(\rho)=9\rho^2-24, \qquad A''(\rho)=18\rho,
\end{align}
we also have
\begin{align}
    \min_{\frac{3}{2}\leq\rho\leq 3} A(\rho) = A\lp\sqrt{\tfrac{8}{3}}\rp > 0.
\end{align}
Thus \eqref{appendix:constant:eqn4} holds for $\frac{3}{2}\leq\rho\leq\sqrt{3}$. On the other hand, if $0\leq\rho\leq\frac{3}{2}$, then
\begin{align}
    A(\rho)+tB(\rho)
    \geq A(\rho)+B(\rho)
    &= 3\rho^3-2\rho^2-15\rho+18,
\end{align}
where we used that $0<t<1$ and $B(\rho)\leq 0$. Again, a calculation shows the expression on the right-hand side is positive on the interval of interest. Returning to \eqref{appendix:constant:eqn2}, we obtain the final estimate
\begin{align}\label{appendix:constant:gammabound}
    \gamma(t)^2
    \leq \frac{2t^2}{27} + t + \frac{5}{3}
    \leq \frac{5}{3}\lp \frac{3}{10}t + 1 \rp^2
    \equiv \Gamma(t).
\end{align}

\subsection{Solving for \texorpdfstring{$I_0,I_\pm$}{the lengths}}\label{appendix:constant:lengths}
We may replace $\gamma(t)$ by $\sqrt{\Gamma(t)}$ in the definition \eqref{def:normalization} of $H(t)$ and then expand the terms in \eqref{def:lengths} to obtain the following expressions
\begin{equation}
\begin{aligned}
    I_0 
    &= 10 \sqrt{\frac{3}{5}} \int_{t_-}^{t_+} \frac{dt}{3t+10} \\
    I_+
    &= 4\sqrt{\frac{5}{3}} \int_{t_+}^b \frac{\lp \sqrt{3}-t\rp^2}{\lp3t+10\rp\lp t+\sqrt{2}\rp^2} \, dt \\
    I_- 
    &= 4\sqrt{15} \lp1-\frac{b}{\sqrt{3}}\rp^2 \int_0^{t_-} \frac{dt}{\lp 3t+10\rp\lp t+\sqrt{2}\rp^2}
\end{aligned}
\end{equation}
Then $I_0$ may be integrated directly to get
\begin{align}
    I_0 
    = 2\sqrt{\frac{5}{3}} \ln\lp\frac{3t_++10}{3t_-+10}\rp.
\end{align}
Apply partial fraction decomposition to $I_\pm$ to find
\begin{align}
    I_+
    &= 4\sqrt{\frac{5}{3}} \ls \frac{C_+}{3}\ln\lp\frac{3b+10}{3t_++10}\rp + D_+ \ln\lp\frac{b+\sqrt{2}}{t_++\sqrt{2}}\rp + \frac{E_+(b-t_+)}{(b+\sqrt{2})(t_++\sqrt{2})} \rs, \\
    I_- 
    &= 4\sqrt{15} \ls1 -\frac{b}{\sqrt{3}}\rp^2 \lp \frac{C_-}{3} \ln\lp\frac{3t_-+10}{10}\rp + D_- \ln\lp\frac{t_-+\sqrt{2}}{\sqrt{2}}\rp + \frac{E_-t_-}{\sqrt{2}(t_-+\sqrt{2})} \rs,
\end{align}
where
\begin{align}
    C_+ = \lp\frac{10+3\sqrt{3}}{10-3\sqrt{2}}\rp^2, \qquad
    E_+ = \frac{5+2\sqrt{6}}{10-3\sqrt{2}}, \qquad
    D_+ = \frac{3-2C_+-10E_+}{10\sqrt{2}}, \\
    C_- = \frac{9}{118-60\sqrt{2}}, \qquad
    E_- = \frac{1}{10-3\sqrt{2}}, \qquad
    D_- = \frac{1-2C_--10E_-}{10\sqrt{2}}.
\end{align}

\printbibliography
\end{document}